\documentclass{article}
\usepackage{amsmath,amsthm,amsfonts,graphicx}
\usepackage{multirow}
\usepackage{slashbox}
\usepackage{multirow}
\def\smallddots{\mathinner{\raise7pt\hbox{.}\raise4pt\hbox{.}\raise1pt\hbox{.}}}
\def\smallsdots{\mathinner{\raise1pt\hbox{.}\raise4pt\hbox{.}\raise7pt\hbox{.}}}
\newtheorem{theorem}{Theorem}
\newtheorem{lemma}[theorem]{Lemma}
\newtheorem{corollary}[theorem]{Corollary}

\theoremstyle{definition}
\newtheorem{definition}[theorem]{Definition}

\newtheorem{example}[theorem]{Example}
\newtheorem{remark}[theorem]{Remark}

\DeclareMathOperator{\diag}{diag}

\DeclareMathOperator{\rank}{rank}
\begin{document}
\title{Transformations of Matrix Structures Work Again  II
\thanks {Some results of this paper have been presented at the 
18th Conference of the International Linear Algebra
Society (ILAS'2013), Providence, RI, 2013
and at the 15th Annual Conference on Computer Algebra in
Scientific Computing (CASC '2003), September 9--13, 2013, Berlin, Germany}}
\author{Victor Y. Pan$^{[1, 2],[a]}$ 
\and\\
$^{[1]}$ Department of Mathematics and Computer Science \\
Lehman College of the City University of New York \\
Bronx, NY 10468 USA \\
$^{[2]}$ Ph.D. Programs in Mathematics  and Computer Science \\
The Graduate Center of the City University of New York \\
New York, NY 10036 USA \\
$^{[a]}$ victor.pan@lehman.cuny.edu \\
http://comet.lehman.cuny.edu/vpan/  \\
} 
\date{}
\maketitle


\begin{abstract}
Matrices with the structures of  Toep\-litz, Han\-kel, Van\-der\-monde and Cauchy types are omnipresent in modern computations in Sciences, Engineering and Signal and Image  Processing. The four matrix classes have distinct features, but in  \cite{P90} we showed that  Van\-der\-monde and Hankel multipliers transform all these structures into each other and proposed to employ this property in order to extend any successful algorithm that inverts matrices of one of these four classes to inverting matrices with the structures of the three other types. The power of this approach was widely recognized later, when novel numerically stable algorithms solved nonsingular Toeplitz linear systems of equations in quadratic (versus classical cubic) arithmetic time based on transforming Toeplitz into Cauchy matrix structures. More recent papers combined such a transformation with a link of the Cauchy matrices to the Hierarchical Semiseparable matrix structure, which is a specialization of matrix representations employed by the Fast Multipole Method. This produced numerically stable algorithms that approximated the solution of a nonsingular Toeplitz linear system of equations in nearly linear arithmetic time. We first revisit the successful method of structure transformation, covering it comprehensively. Then we analyze the latter efficient approximation algorithms for Toeplitz linear systems and extend them to approximate the products of Van\-der\-monde and Cauchy matrices by a vector and the solutions of Van\-der\-monde and Cauchy linear systems of equations where they are nonsingular and well conditioned. We decrease the arithmetic cost of the known  numerical approximation algorithms for these tasks from quadratic to nearly linear, and similarly for the computations with the matrices of a more general class having structures of Van\-der\-monde and Cauchy types and for polynomial and rational evaluation and interpolation. We also accelerate a little further the known numerical approximation algorithms for a nonsingular Toeplitz or Toeplitz-like linear system by employing distinct transformations of matrix structures, and we briefly discuss some natural research challenges, particularly some promising applications of our techniques to high precision computations.
\end{abstract}



\section{Introduction}\label{s1}

\begin{table}[ht]
\caption{Four classes of structured matrices}
\label{t1}.
\begin{center}
\renewcommand{\arraystretch}{1.2}
\begin{tabular}{c|c}
Toeplitz matrices $T=\left(t_{i-j}\right)_{i,j=0}^{n-1}$&Hankel matrices $H=\left(h_{i+j}\right)_{i,j=0}^{n-1}$ \\
$\begin{pmatrix}t_0&t_{-1}&\cdots&t_{1-n}\\ t_1&t_0&\smallddots&\vdots\\ \vdots&\smallddots&\smallddots&t_{-1}\\ t_{n-1}&\cdots&t_1&t_0\end{pmatrix}$&$\begin{pmatrix}h_0&h_1&\cdots&h_{n-1}\\ h_1&h_2&\smallsdots&h_n\\ \vdots&\smallsdots&\smallsdots&\vdots\\ h_{n-1}&h_n&\cdots&h_{2n-2}\end{pmatrix}$ 
\\ \hline 


Van\-der\-monde matrices 
$V=V_{\bf s}=\left(s_i^{j}\right)_{i,j=0}^{n-1}$&Cauchy matrices $C=C_{\bf s,t}=\left(\frac{1}{s_i-t_j}\right)_{i,j=0}^{n-1}$ \\
$\begin{pmatrix}1&s_1&\cdots&s_1^{n-1}\\ 1&s_2&\cdots&s_2^{n-1}\\ \vdots&\vdots&&\vdots\\ 1&s_{n}&\cdots&s_{n}^{n-1}\end{pmatrix}$&$\begin{pmatrix}\frac{1}{s_1-t_1}&\cdots&\frac{1}{s_1-t_{n}}\\ \frac{1}{s_2-t_1}&\cdots&\frac{1}{s_2-t_{n}}\\ \vdots&&\vdots\\ \frac{1}{s_{n}-t_1}&\cdots&\frac{1}{s_{n}-t_{n}}\end{pmatrix}$
\end{tabular}
\end{center}
\end{table}

Table \ref{t1} displays four classes of most popular structured matrices,
which
are omnipresent in modern computations for Sciences, Engineering, and Signal
and Image Processing and
which have been naturally extended to larger classes of matrices, 
$\mathcal T$, $\mathcal H$, $\mathcal V$, and $\mathcal C$,
having structures of Toep\-litz, Hankel,
Van\-der\-monde and Cauchy types,
 respectively. 
Such matrices 
can be readily expressed via
their displacements of small ranks, which implies
their further attractive properties:

\begin{itemize}\itemsep=-0.6mm\vspace{-2mm}
\item
Compressed representation of a matrices as well as their products and inverses through 
a small number of parameters 
\item
Multiplication by vectors in nearly linear arithmetic time
\item
Solution of nonsingular
linear systems of equations with these matrices
in quadratic or nearly linear arithmetic time
\vspace{-2mm}
\end{itemize}

Highly successful research and implementation work based
on these properties has been continuing for more than three decades.
We follow \cite{P90}
and employ structured matrix multiplications to
 transform the four matrix structures
into each other. For example, 
$\mathcal T \mathcal H=\mathcal H\mathcal T=\mathcal H$, 
$\mathcal H\mathcal H=\mathcal T$, whereas $V^TV$ is a Hankel matrix.
The paper \cite{P90} showed 
that {\em such techniques enable one
  to extend any successful algorithm for  
the
inversion 
of the matrices of any
  of the four classes 
 $\mathcal T$, $\mathcal H$, 
$\mathcal V$, and $\mathcal C$
to the  matrices of the  
three other classes, and similarly for the solution of linear systems
of equations}.
We cover this approach comprehensively
and simplify its presentation versus  \cite{P90} because 
we  
employ Sylvester's displacements $AM-MB$, 
versus Stein's displacements $M-AMB$ in \cite{P90},
and apply the machinery of operating with them from \cite{P00} and 
 \cite[Section 1.5]{P01}. 
We study quite simple structure transforms, 
but they have surprising power where
they link together
  matrix classes
having distinct
features.

For example, row and column interchanges destroy Toeplitz and Hankel
but not Cauchy matrix structure, and
 \cite{GKO95} and \cite{G98} obtained numerically stable solution 
of  Toeplitz and Hankel linear systems without pivoting 
by transforming the inputs into
Cauchy-like  matrices. The resulting algorithms run in
 quadratic arithmetic time versus classical cubic. 
Like  \cite{GKO95} and \cite{G98} the papers 
\cite{MRT05},  \cite{CGS07}, \cite{XXG12},   and  \cite{XXCB}
reduced the solution of a nonsingular  Toeplitz linear system
of $n$ equations to computations with
a special Cauchy matrix $C=C_{\bf s,t}$ whose 
 $2n$ knots $s_0,t_0,\dots,s_{n-1},t_{n-1}$ 
are equally spaced on the unit circle 
$\{z:~|z|=1\}$, but then the authors
applied a variant of
fast numerically stable FMM to compute
 HSS compressed approximation of this matrix and consequently to yield
approximate solution of the original task in 
 nearly linear arithmetic time.
``HSS" and ``FMM" are the acronyms for ``Hierarchically Semiseparable"
and ``Fast Multipole Method", respectively.
``Historically HSS representation
is just a special case of the representations commonly exploited in the 
FMM literature" \cite{CDGLP06}. 
We refer the reader to the papers 
 \cite{CGR98}, 
 \cite{GR87},  \cite{DGR96}, 
 \cite{BY13} and the bibliography therein
on FMM and to
\cite{B10}, \cite{CDGLP06},
\cite{CGS07},
\cite{DV98},
 \cite{GKK85}, \cite{T00},  \cite{X13},
  \cite{XXG12}, \cite{XXCB},
 and the bibliography therein
on HSS matrices and their link to FMM.

We analyze the fast
numerically stable algorithms
of \cite{MRT05},  \cite{CGS07}, \cite{XXG12},   and  \cite{XXCB},
which treat the cited special Cauchy matrix
and extend these algorithms to treat 
 a quite general subclass of Cauchy  matrices,
which includes  {\em CV} and {\em CV-like} matrices,
 obtained by FFT-based transforms from
Van\-der\-monde matrices and their transposes. 
The known approximation algorithms run in quadratic  arithmetic time 
even for multiplication of these matrices by a vector, whereas 
we yield 
 nearly linear arithmetic time both for 
that task and  computing 
approximate solutions of linear systems of equations with
these matrices  where they are nonsingular and well conditioned. 
The solutions are immediately extended 
to the computations with Van\-der\-monde matrices and 
to polynomial and rational evaluation and interpolation. 
In the cases of solving linear systems and interpolation,
the power of our numerical algorithms is limited
because 
only Van\-der\-monde matrices of a  narrow although important subclass 
   are well conditioned
(see \cite{GI88}), and we prove a similar property for the
CV matrices (see  our Remark \ref{renmr}).

In Section \ref{stls} we employ another kind of transformations of matrix structures, which 
we call {\em functional}, to extend the power of the fast numerically stable 
algorithm of \cite{DGR96}, proposed for polynomial evaluation at a set of real knots. This enables
us to accelerate a little further the approximation algorithms of 
 \cite{MRT05},  \cite{CGS07}, \cite{XXG12},   and  \cite{XXCB} 
for Toeplitz linear systems.
 By means of other transformations of matrix structures
we extend the approximation algorithms from CV matrices to  
Cauchy and Cauchy-like matrices
with arbitrary sets of knots, but point out potential numerical limitations
of these results. 
At the end of the paper we discuss
some specific directions to the acceleration of our proposed 
approximation algorithms by logarithmic factor and 
to a more significant acceleration
in the case of high precision computations. Further research could reveal 
new transformations of matrix structures 
with significant algorithmic applications. 

Besides new demonstration of the power of the transformation techniques, 
our   analysis of the 
approximation algorithms of
 \cite{MRT05},  \cite{CGS07}, \cite{XXG12},   and  \cite{XXCB}
can be of some technical interest because instead of the special Cauchy matrix
used in these papers we  cover CV matrices, which
have 
one 
of their two knot sets $\{s_0,\dots,s_{n-1}\}$ or $\{t_0,\dots,t_{n-1}\}$
equally spaced on the unit circle 
$\{z:~|z|=1\}$,
whereas the remaining $n$ knots 
are arbitrary. We 
still obtain  
a desired HSS representation by partitioning the knots
according to the angles in their
polar coordinates. Our study provides
 a new insight into the subject, and
 we prove that the admissible blocks
of the $n\times n$ HSS 
$\epsilon$-approximations of CV matrices have ranks of order $O(\log (n/\epsilon))$,
which decrease to $O(\log (1/\epsilon))$ in the case of
approximation of the special Cauchy matrix linked to  Toeplitz  
inputs.

We organize our presentation as follows.
After recalling some definitions and basic facts
on general  matrices  and on four classes of structured matrices
$\mathcal T$, $\mathcal H$, $\mathcal V$, and $\mathcal C$
in the next three sections, we cover in some detail the 
transformations of matrix structures among these classes
in Section \ref{sdtr}. That study is not used in the second part
of our paper 
(Sections  \ref{shss}--\ref{shssapr}), where
we approximate Cauchy matrices by HSS matrices.
Namely we define HSS matrices in Section \ref{shss},
estimate numerical ranks of Cauchy and Cauchy-like
 matrices of a large class
in Section \ref{snrqs}, and extend these estimates to 
compute the HSS type approximations of these matrices in 
Section \ref{shssapr}. In Section \ref{svmv}
we combine the results of the two parts
as well as some functional transformations of matrix structures
to devise approximation
algorithms for computations with
 various structured matrices.
 We conclude the paper with Section \ref{scnc}.
For simplicity we assume square structured matrices throughout,
but our study can be readily extended to the case 
of rectangular matrices (cf. \cite{Pa}).

\section{Some definitions and basic facts}\label{sdef}

Hereafter ``flop" stands for ``arithmetic operation
with real or complex numbers"; 
 the concepts ``large", ``small", ``near", ``close", ``approximate", 
``ill  conditioned" and ``well conditioned" 
are 
quantified in the context. 
Next we recall and extend some basic definitions  
and facts 
on computations with general and structured  matrices
(cf. \cite{GL96}, \cite{S98}, \cite{P01}).


\subsection{General matrices}\label{sgen}



$M=(m_{i,j})_{i,j=1}^{m,n}$ is an $m\times n$ matrix,
$M^T$ and $M^H$ are its transpose 
and  Hermitian 
transpose,
respectively. 
$M^{-T}=(M^T)^{-1}=(M^{-1})^T$.
$(B_0~|~\dots~|~B_{k-1})$ and $(B_0~\dots~B_{k-1})$ denote a $1\times k$ block matrix with the blocks $B_0,\dots,B_{k-1}$.
$\Sigma= \diag(\Sigma_0,\dots,\Sigma_{k-1})=\diag(\Sigma_j)_{j=0}^{k-1}$ is a $k\times k$ 
block diagonal matrix with 
the diagonal blocks $\Sigma_0,\dots,\Sigma_{k-1}$, possibly rectangular. 
${\bf e}_1,\dots,{\bf e}_n$ are the $n$ coordinate vectors
of a dimension $n$.
${\bf s}=(s_j)_{j=0}^{n-1}=\sum_{i=0}^{n-1}s_i{\bf e}_i$.
 $D_{\bf s}=\diag({\bf s})=\diag(s_i)_{i=0}^{n-1}$.
  $I=I_n=({\bf e}_1~|~\dots~|~{\bf e}_n)$
and  $J=J_n=({\bf e}_n~|~\dots~|~{\bf e}_1)$ are
the $n\times n$ identity and 
  reflection matrices, respectively.
 $J=J^T=J^{-1}$.

{\bf Preprocessors.} For three nonsingular matrices $P$, $M$, and $N$
 and a vector ${\bf b}$,
it holds that
\begin{equation}\label{eqprepr}
M^{-1}=N(PMN)^{-1}P,~~~ 
PMN{\bf y}=P{\bf b},~~{\bf x}=N{\bf y}. 
\end{equation}

{\bf Generators.}
Given an $m\times n$ matrix $M$ of a rank $r$ 
and an integer $l\ge r$, we have a
nonunique expression
$M=FG^T$ for pairs $(F,G)$ of matrices of sizes 
$m\times l$ and $n\times l$, respectively. 
We call such a pair $(F,G)$
a {\em generator of length} $l$ for the matrix $M$.


{\bf Norm, conditioning, orthogonality, numerical rank.}
$||M||=||M||_2$ is the spectral norm of an $n\times n$ matrix $M=(m_{i,j})_{i,j=0}^{n-1}$.
We write $|M|=\max_{i,j=0}^{n-1}|m_{i,j}|$. $||M||=||M^H||\le \sqrt{mn}~|M|$.
For a fixed tolerance $\tau$, 
the $\tau$-{\em rank} of a matrix $M$
is the minimum rank of matrices 
in its $\tau$-neighborhood,
$\{W:~|W-M|\le \tau\}$.
The  {\em numerical rank} of a matrix 
is its $\tau$-rank for a small positive $\tau$.
A matrix is {\em ill conditioned} 
if its rank exceeds its numerical rank.
A matrix $M$ is {\em unitary} or {\em orthogonal} if $M^HM=I$ or $MM^H=I$.
It is {\em quasi\-unitary} if $cM$ is unitary for a nonzero constant $c$.
A vector ${\bf u}$ is unitary if and only if $||{\bf u}||=1$,
and if so, we call it a {\em unit vector}.

\subsection{DFT and $f$-circulant matrices}\label{sdft}  

Even for moderately large integers $n$ the entries of  an $n\times n$ Van\-der\-monde matrix $V_{\bf s}$ 
 vary in magnitude greatly unless $|s_i|\approx 1$ for all $i$. Next we cover
the quasiunitary Van\-der\-monde matrices $\Omega$ and $\Omega^H$ 
and the related class of $f$-circulant matrices
 (cf. \cite [Section 3.4]{BP94}, \cite [Section 2.3 and 2.6]{P01}).

 $\omega_n={\rm exp}(\frac{2\pi}{n}\sqrt{-1})$  denotes 
a primitive $n$th root of $1$. 
Its powers $1,\omega_n,\dots,\omega_n^{n-1}$
are equally spaced on
the unit circle $\{z:|z|=1\}$.
$\Omega=\Omega_n=(\omega_n^{ij})_{i,j=0}^{n-1}$ denotes the $n\times n$ matrix of 
{\em DFT}, that is  of the 
{\em discrete Fourier transform} at $n$ points. 
It holds that $\Omega \Omega^H=nI$, and so
 $\Omega$, $\Omega^H$, 
and $\Omega^{-1}=\frac{1}{n}\Omega^H$
are quasi\-unitary matrices,
whereas
  $\frac{1}{\sqrt n}\Omega$ and $\frac{1}{\sqrt n}\Omega^H$
are unitary matrices.

$Z_f=\begin{pmatrix} {\bf 0}^T & f\\ I_{n-1} & {\bf 0}\end{pmatrix}$
is the  $n\times n$ matrix of $f$-circular shift
for a scalar $f$,
\begin{equation}\label{eqrever}
JZ_fJ=Z_{f}^T,~~
JZ_f^TJ=Z_{f} 
\end{equation}
for any pairs of scalars $e$ and $f$, and if $f\neq 0$, then
\begin{equation}\label{eqinv}
Z_f^{-1}=Z_{1/f}^T.
\end{equation}

$Z_f({\bf v})=\sum_{i=0}^{n-1}v_{i}Z_f^{i}$
is an  $f$-circulant matrix,
defined by its first column ${\bf v}=(v_i)_{i=0}^{n-1}$
and a scalar $f\neq 0$ and called circulant for $f=1$. It is
 a Toeplitz matrix  
and can be called a {\em DFT-based Toeplitz matrix}
in view of the following results.

\begin{theorem}\label{thcpw} (See \cite{CPW74}.)
It holds that
$Z_{f^n}({\bf v})=V_f^{-1}D(V_f{\bf v})V_f$
provided that $f\ne 0$,
$\Omega=(\omega_n^{ij})_{i,j=0}^{n-1}$ is the $n\times n$ matrix of DFT,
$D({\bf u})=\diag(u_i)_{i=0}^{n-1}$ for a vector ${\bf u}=(u_i)_{i=0}^{n-1}$,  
and $V_f=\Omega\diag(f^{i})_{i=0}^{n-1}$ is  the matrix of (\ref{eqvf}). In particular 
$Z_1({\bf v})=\Omega^{-1}D(\Omega{\bf v})\Omega.$
\end{theorem}

\subsection{Cauchy and Van\-der\-monde  matrices}\label{scvm} 

Recall the following properties of Cauchy and Van\-der\-monde  matrices (cf. \cite[Chapters 2 and 3]{P01}),

\begin{equation}\label{eqctr}
C_{\bf s,t}=-C_{\bf t,s}^T,
\end{equation}
\begin{equation}\label{eqfhr}
C_{\bf s,t}=\diag(t(s_i)^{-1})_{i=0}^{n-1}V_{\bf s}V^{-1}_{\bf t}\diag(t'(t_j))_{j=0}^{n-1}
\end{equation}
where ${\bf s}=(s_i)_{i=1}^{n}$, 
${\bf t}=(t_j)_{j=0}^{n-1}$, and
$t(x)=\prod_{j=0}^{n-1}(x-t_j)$.

\begin{theorem}\label{thcvdt} 
$\det (V)=\prod_{i>k}(s_i-s_k)$ and 
$\det (C)=\prod_{i<j}(s_j-s_i)(t_i-t_j)/\prod_{i,j}(s_i-t_j)$
\end{theorem}

\begin{corollary}\label{conscv} 
The matrices $V$ and $C$ of Table \ref{t1}
are nonsingular where all $2n$ scalars $s_0$,$\dots$,$s_{n-1}$, $t_0$,$\dots$,$t_{n-1}$
are distinct. 
\end{corollary}

\begin{theorem}\label{thcv1} 
 A row interchange preserves both Van\-der\-monde	  
and Cauchy structures. A column interchange preserves 	  
Cauchy structure. 
\end{theorem}


Equations (\ref{eqctr}) and (\ref{eqfhr}) 
 link together Cauchy and
  Van\-der\-monde matrices and their transposes. 
Next we simplify these links.
Write
\begin{equation}\label{eqvf}
V_f=((f\omega_n^{i})^{j})_{i,j=0}^{n-1}=\Omega\diag(f^{j})_{j=0}^{n-1},
\end{equation}
$$C_{{\bf s},f}=\Bigg(\frac{1}{s_i-f\omega_n^{j}}\Bigg)_{i,j=0}^{n-1},~
C_{e,{\bf t}}=\Bigg(\frac{1}{e\omega_n^{i}-t_j}\Bigg)_{i,j=0}^{n-1},~
C_{e,f}=\Bigg(\frac{1}{e\omega_n^{i}-f\omega_n^{j}}\Bigg)_{i,j=0}^{n-1}$$
and observe that $\Omega=V_1$,
$\Omega^H=V_{1}^{-1}$,
and {\em the matrices} $V_f$ {\em are quasi\-unitary} 
where $|f|=1$. 

For ${\bf t}=(f\omega_n^{j})_{j=0}^{n-1}$, 
it holds that $t(x)=x^n-f^n$, $t'(x)=nx^{n-1}$, $t(s_i)=s_i^n-f^n$,
$t'(t_j)=nf^{n-1}\omega_n^{-j}$  for all $j$,
and 
$nV_f^{-1}=\diag(f^{-i})_{i=0}^{n-1}\Omega^H$.
Substitute these equations  into (\ref{eqfhr}) 
and obtain

\begin{equation}\label{eqfhrf}
C_{{\bf s},f}=
\diag\Big(\frac{f^{n-1}}{s_i^{n}-f^n}\Big)_{i=0}^{n-1}V_{\bf s}\diag(f^{-j})_{j=0}^{n-1}\Omega^H\diag(\omega_n^{-j})_{j=0}^{n-1},
\end{equation}

$$C_{e,f}=\frac{f^{n-1}}{e^n-f^{n}}\Omega\diag((e/f)^{i})_{i=0}^{n-1}\Omega^H\diag(\omega_n^{-j})_{j=0}^{n-1},$$

\begin{equation}\label{eqvs}
V_{\bf s}=\frac{f^{1-n}}{n}\diag\Big(s^n_i-f^n\Big)_{i=0}^{n-1}C_{{\bf s},f}\diag(\omega_n^{j})_{j=0}^{n-1}\Omega\diag(f^{j})_{j=0}^{n-1},
\end{equation}
\begin{equation}\label{eqvst}
V_{\bf s}^T=-\frac{f^{1-n}}{n}\diag(f^{j})_{j=0}^{n-1}\Omega \diag(\omega_n^{j})_{j=0}^{n-1}C_{f,{\bf s}}\diag(s^n_i-f^n)_{i=0}^{n-1},
\end{equation}
\begin{equation}\label{eqvs-}
V_{\bf s}^{-1}=n\diag(f^{-j})_{j=0}^{n-1}\Omega^H\diag(\omega_n^{-j})_{j=0}^{n-1}C_{{\bf s},f}^{-1}\diag\Big(\frac{f^{n-1}}{s^n_i-f^n}\Big)_{i=0}^{n-1},
\end{equation}
\begin{equation}\label{eqvs-t}
V_{\bf s}^{-T}=- n\diag\Big(\frac{f^{n-1}}{s^n_i-f^n}\Big)_{i=0}^{n-1}C_{f,{\bf s}}^{-1}\diag(\omega_n^{-j})_{j=0}^{n-1}\Omega^H\diag(f^{-j})_{j=0}^{n-1}.
\end{equation}

\begin{definition}\label{defcvdft} 
Hereafter we refer to 
the matrices 
$V_f$, $C_{{\bf s},f}$, $C_{e,{\bf t}}$, and $C_{e,f}$
for two scalars $e$ and $f$
as {\em FV, FC, CF}, and {\em FCF matrices}, respectively. We refer  to
the matrices 
$C_{{\bf s},f}$ and $C_{e,{\bf t}}$ as
{\em CV matrices}
and to 
 the FV matrices $V_f$ and the FCF matrices $C_{e,f}$
as
the {\em DFT-based} matrices.
\end{definition}

Similarly to the DFT matrix $\Omega$,
the
DFT-based matrices 
have
their basic sets of knots
$\mathbb S=\{s_1,\dots,s_n\}$
and
$\mathbb T=\{t_1,\dots,t_n\}$
equally spaced on the unit circle $\{z:~|z|=1\}$, and 
equation (\ref{eqfhrf}) links the
 CV matrices
to Van\-der\-monde matrices.
In spite of all these links Cauchy and Van\-der\-monde
matrices also have very distinct features 
(cf. Remark \ref{revndsc}).

Finally  \cite[equation (3.4.1)]{P01}
  links a Van\-der\-monde matrix and its transpose  as follows,

\begin{equation}\label{eqvtinv}
V_{\bf t}JZ_f({\bf w}+f{\bf e}_1)V_{\bf t}^T=\diag(t'(t_i)(f-t_i^n))_{i=0}^{n-1}~{\rm for~any~scalar}~f.
\end{equation}
Here $t(x)=\prod_{j=0}^{n-1}(x-t_j)$ and $w(x)=t(x)-x^n$
are polynomials with
the coefficient vectors  ${\bf t}$ and ${\bf w}$, respectively
(see Theorem \ref{thevint} on their evaluation 
and Example \ref{exln}
on their approximation).

\subsection{The complexity of computations with DFT, Toep\-litz, Hankel,
Cauchy and Van\-der\-monde  matrices}\label{scmpl}

We begin with the following observation.

\begin{theorem}\label{thth}
If $T$ is a Toep\-litz matrix, then $TJ$ and $JT$ are Hankel matrices, whereas
If $H$ is a  Hankel matrix, then $HJ$ and $JH$ are Toep\-litz matrices.
\end{theorem}

We also recall the 
following  results
(see, e.g., \cite[Sections 1.2 and 3.4]{BP94}
on their proof  and on the numerical stability 
of the supporting algorithms).

\begin{theorem}\label{thdft}
For any vector ${\bf v}$ of dimension $n$
one can compute the vectors $\Omega{\bf v}$
and $\Omega^{-1}{\bf v}$ by using $O(n\log (n))$
flops. If $n$ 
 is a power of 2, then
one can  compute the vectors $\Omega{\bf v}$
and $\Omega^{-1}{\bf v}$
by applying {\em FFT}, that is by using
$0.5n\log_2 (n)$
 and 
$0.5n\log_2 (n)+n$ flops,
respectively.
\end{theorem}

 Theorems \ref{thcpw},  \ref{thth}, and \ref{thdft} combined
with various  techniques of matrix computations,
imply the following results
(cf. \cite[Chapter 2 and 3]{P01}). 

\begin{theorem}\label{thcvdft} 
$O(n\log^h (n))$  flops are sufficient to compute the product 
of an $n\times n$ matrix $M$ and a vector
${\bf u}$ where $h=1$
if $M$ is a Toeplitz or Hankel matrix and $h=2$
if  $M$ is a Van\-der\-monde matrix,  its transpose, or a Cauchy matrix. 
$O(n\log^2 (n))$  flops are sufficient to compute the
solution ${\bf x}$ of a nonsingular linear system of $n$ equations $M{\bf x}={\bf u}$
with any of such  matrices $M$.
\end{theorem}

The algorithms supporting this theorem  
are numerically stable where the matrix $M$ is DFT-based
(combine Theorems  \ref{thcpw} and \ref{thdft} and equation (\ref{eqvf}))
and  where we multiply
a Toeplitz or Hankel matrix $M$ by a vector
(embed an $n\times n$ Toeplitz matrix into $(2n-1)\times (2n-1)$ circulant matrix 
and then combine Theorems  \ref{thcpw}, \ref{thth}, and \ref{thdft}).
Otherwise the algorithms have numerical stability problems,
and for numerical computations the users employ quadratic
arithmetic time algorithms  \cite{BEGO08}, \cite{BF00}, \cite{KZ08},
in spite of substantial research progress 
reported in the papers
 \cite{PRT92}, \cite{PSLT93},  \cite{P95}, \cite{PZHY97},
and particularly \cite{DGR96}, 
which applied a 1-dimensional  adaptive FMM
using Lagrange interpolation at Chebyshev's knots
to prove
the following result.
\begin{theorem}\label{thdgr}  (Cf. \cite[Sections 3 and 4]{DGR96}.)
Assume  a positive  $\epsilon<1$, a unit vector ${\bf u}$, and
an $n\times n$ Cauchy matrix $C_{\bf s,t}$ with real knots $s_i$ and $t_j$.
Then some numerically stable algorithms 
use  $O(n\log(1/\epsilon))$ flops to
approximate 
within the norm bound $\epsilon$
the product  $C_{\bf s,t}{\bf u}$
 and if the matrix is nonsingular, then also
the solution ${\bf x}$ of the linear system $C_{\bf s,t}{\bf x}={\bf u}$.
\end{theorem}


\section{The structures of Toeplitz, 
Hankel, Van\-der\-monde and Cauchy
types. Displacement ranks and generators}\label{sfourd}

We generalize
 the four classes of
matrices of Table \ref{t1}
by employing
the Sylvester displacements
$AM-MB$ where 
the pair of {\em operator matrices} $A$ and $B$
is associated with a fixed matrix structure.  
(See \cite[Theorem 1.3.1]{P01} on a simple link to the Stein displacements  
$M-AMB$.)
The rank and the
generators of the displacement
of a matrix $M$ (for a fixed  operator matrices $A$ and $B$
and  tolerance $\tau$)
are said to be the {\em displacement rank} (denoted $d_{A,B}(M)$)
 and the {\em displacement generators},  
of the matrix $M$,
respectively (cf. \cite{KKM79}), \cite{P01}, \cite{BM01}).

\begin{definition}\label{def4t}
If the displacement rank of a matrix is small 
(in context) for a pair of 
operator matrices 
associated with 
Toeplitz, Hankel, Van\-der\-monde, transpose of  Van\-der\-monde or Cauchy matrices
in Theorem \ref{thdsplr}
below,
then the matrix is said to have the {\em structure
of Toeplitz, Hankel,  Van\-der\-monde,  transposed Van\-der\-monde or Cauchy type},
respectively. Hereafter 
 $\mathcal T$, 
$\mathcal H$,
$\mathcal V$,
$\mathcal V^T$,
 and
$\mathcal C$ denote the five classes of these
 matrices (cf. Table \ref{taboperm}).
The classes $\mathcal V$,
$\mathcal V^T$,
 and
$\mathcal C$ consist of distinct subclasses 
$\mathcal V_{\bf s}$,
$\mathcal V^T_{\bf s}$,
 and
$\mathcal C_{\bf s,t}$
defined by the vectors ${\bf s}$
and  ${\bf t}$ and the operator matrices
$D_{\bf s}$ and
$D_{\bf t}$, respectively,
or equivalently by the  bases $V_{\bf s}$
and
$C_{\bf s,t}$ of these subclasses.
To simplify the notation
we will sometimes drop the subscripts
${\bf s}$ and ${\bf t}$
where they are not important 
or are defined by context.
\end{definition}

\begin{definition}\label{def4tf} (Cf. Definition \ref{defcvdft}.)
Define the matrix classes
$\mathcal {FV}=
\cup_e\mathcal V_e$,
$\mathcal {FC}=
\cup_f\mathcal C_{{\bf s},f}$,
$\mathcal {CF}=
\cup_e\mathcal C_{e,{\bf t}}$,
and $\mathcal {FCF}
=\cup_{e,f}\mathcal C_{e,f}$
where the unions are over all complex scalars $e$ and $f$.
These matrix classes extend
 the classes of FV, FC, CF, and FCF matrices, 
respectively. We also define the classes $\mathcal {CV}$
(extending the CV matrices) 
and
$\mathcal {V^TF}=\cup_e\mathcal V_e^T$.
We say that the above matrix classes consist of  $FV$-like,
$V^TF$-like, $FC$-like, $CF$-like, $FCF$-like, and $CV$-like
 matrices, which have structures
of $\mathcal {FV}$-type, $\mathcal {V^TF}$-type,
$\mathcal {FC}$-type, $\mathcal {CF}$-type,
 $\mathcal {FCF}$-type, and  $\mathcal {CV}$-type,
respectively.
\end{definition}

In our Theorems \ref{thdsplr} and \ref{thdexpr}  we write 
$(t), ~(h),~ (th),~ (v), ~(v^T)$, and $(c)$ to indicate the matrix structures of
Toeplitz, 
Hankel, 
Toeplitz or 
Hankel,
 Van\-der\-monde, transposed Van\-der\-monde, and Cauchy
types, respectively.
Recall the following well known results.
\begin{theorem}\label{thdsplr} {\em Displacements of basic structured matrices.}

(th) For a pair of scalars $e$ and $f$ and
 two matrices $T$ (Toeplitz)
and $H$ (Hankel) of Table \ref{t1},
 the following displacements have ranks at most $2$
(see some expressions for the shortest displacement 
generators in \cite[Section 4.2]{P01}),
$$Z_eT-TZ_f,~
Z_e^TT-TZ_f^T,~Z_e^TH-HZ_f~{\rm and}~Z_eH-HZ_f^T.$$ 

(v) For a scalar $e$ and a
Vandermonde matrix $V$ of Table \ref{t1} we have 
\begin{equation}\label{eqvand}
VZ_e=D_{\bf s}V-(s_i^n-e)_{i=0}^{n-1}{\bf e}_n^T,
\end{equation}
\begin{equation}\label{eqvandt}
Z_e^TV^T=V^TD_{\bf s}-{\bf e}_n((s_i^n-e)_{i=0}^{n-1})^T,
\end{equation}
and so
the displacements 
$D_{\bf s}V-VZ_{e}$ and $Z_{e}^TV^T-V^TD_{\bf s}$
either vanish if
$s_i^n=e$ for $i=0,\dots,n-1$
or 
have rank $1$
otherwise.

\medskip

(c) For 
 two vectors 
${\bf s}=(s_i)_{i=0}^{n-1}$ and
${\bf t}=(t_i)_{i=0}^{n-1}$, 
 a 
 Cauchy matrix $C$ of Table \ref{t1},
and the vector ${\bf e}=(1,\dots,1)^T$  of dimension $n$ filled with ones,
it holds that
\begin{equation}\label{eqcch}
D_{\bf s}C-CD_{\bf t}={\bf e}{\bf e}^T,
~~\rank(D_{\bf s}C-CD_{\bf t})=1.
\end{equation}
\end{theorem}


\begin{theorem}\label{thzef} 
For two scalars $e$ and $f$
and five matrices $A$, $B$, $C$, $D$, and $M$ we have
$d_{C,D}(M)-d_{A,B}(M)\le 1$ 
where either $A=C$, $B=Z_e$, $D=Z_f$ or $A=C$, $B=Z_e^T$, $D=Z_f^T$
and similarly where either 
 $B=D$, $A=Z_e$, $C=Z_f$ or $B=D$, $A=Z_e^T$, $C=Z_f^T$.
\end{theorem} 
\begin{proof} 
The matrix $Z_e-Z_f=(e-f){\bf e}_1{\bf e}_n^T$ 
has rank at most $1$
for any pair of scalars $e$ and $f$.
Therefore the matrices 
$(Z_eM-MB)-(Z_fM-MB)=Z_eM-Z_fM=(Z_e-Z_f)M$ and 
$(AM-MZ_e)-(AM-MZ_f)=-M(Z_e-Z_f)$
have ranks at most $1$.
\end{proof}

The theorem implies that
the classes  $\mathcal T$, 
$\mathcal H$,
$\mathcal V$, and
$\mathcal V^T$ stay intact
when we vary the scalars
$e$ and $f$, defining 
the operator matrices $Z_e$ and $Z_f$.

Table \ref{taboperm} displays the pairs of operator matrices
associated with the matrices of the seven classes $\mathcal T$, 
$\mathcal H$,
$\mathcal V_{\bf s}$,
$\mathcal V^{-1}_{\bf s}$,
$\mathcal V^{T}_{\bf s}$,
$\mathcal V^{-T}_{\bf s}$,
and
$\mathcal C_{\bf s,t}$.
Five of these classes  are employed in Theorems \ref{thdsplr} and \ref{thdexpr}.
 $\mathcal V^{-1}_{\bf s}$
and
$\mathcal V^{-T}_{\bf s}$
denote 
the 
classes of 
 the inverses and the 
 transposed
inverses 
of the matrices
of the class $\mathcal V_{\bf s}$, respectively.
Equation (\ref{eqinvm}) of the next section enables us to
express their associated operator matrices 
through the ones for 
the classes $\mathcal V_{\bf s}$ and $\mathcal V^T_{\bf s}$. 

The following theorem  
 expresses the $n^2$ entries of an $n\times n$ matrix $M$
 through the $2dn$ entries of its displacement 
generator $(F,G)$ defined
under the operator matrices of Theorem \ref{thdsplr}
and Table \ref{taboperm}.
See some  other expressions  
for various classes  of structured matrices 
through their generators
in 
\cite{GO94}, \cite[Sections 4.4 and 4.5]{P01}, and 
 \cite{PW03}.

\begin{theorem}\label{thdexpr}
Suppose $s_0,\dots,s_{n-1},t_0,\dots,t_{n-1}$ are  $2n$
distinct scalars, 
${\bf s}=(s_k)_{k=0}^{n-1}$, ${\bf t}=(t_k)_{k=0}^{n-1}$,
$V=(s_i^{k-1})_{i,k=0}^{n-1}$,   $C=(\frac{1}{s_i-t_k})_{i,k=0}^{n-1}$,
$e$ and $f$ are two distinct scalars,
${\bf f}_1, \dots, {\bf f}_d,{\bf g}_1, \dots, {\bf g}_d$ are $2d$ vectors of dimension $n$,
 ${\bf u}_1, \dots, {\bf u}_n,{\bf v}_1, \dots, {\bf v}_n$ are $2n$ vectors of dimension $d$,
and $F$ and $G$ are $n\times d$ matrices such that  
$F=\begin{pmatrix}{\bf u}_1\\ \vdots\\ {\bf u}_n\end{pmatrix}=({\bf f}_1~|~\cdots~|~{\bf f}_d),~~
G=\begin{pmatrix}{\bf v}_1\\ \vdots\\ {\bf v}_n\end{pmatrix}=({\bf g}_1~|~\cdots~|~{\bf g}_d)$.
 Then 

(t) $(e-f)M=\sum_{j=1}^dZ_e({\bf f}_j)Z_f(J{\bf g}_j)$
if $Z_eM-MZ_f=FG^T$, $e\neq f$;

$(e-f)M=\sum_{j=1}^dZ_e(J{\bf f}_j)^TZ_f({\bf g}_j)^T=J\sum_{j=1}^dZ_e(J{\bf f}_j)Z_f({\bf g}_j)J$
if $Z_e^TM-MZ_f^T=FG^T$, $e\neq f$,

(h) $(e-f)M=\sum_{j=1}^dZ_e({\bf f}_j)Z_f({\bf g}_j)J$
if $Z_eM-MZ_f^T=FG^T$,  $e\neq f$; 

$(e-f)M=J\sum_{j=1}^dZ_e(J{\bf f}_j)Z_f(J{\bf g}_j)^T$
if $Z_e^TM-MZ_f=FG^T$,  $e\neq f$,

(v) $M=
\diag(\frac{1}{s_i^n-e})_{i=0}^{n-1}\sum_{j=1}^d\diag({\bf f}_j)VZ_{e}(J{\bf g}_j)$
if  
$D_{\bf s}M-MZ_{e}=FG^T$
and if
$s_i^n\neq e$ for $i=0,\dots,n-1$;  

(v$^T$) $M=\diag(\frac{1}{e-s_i^n})_{i=0}^{n-1}\sum_{j=1}^dZ_{e}(J{\bf f}_j)^TV^T\diag({\bf g}_j)$
if  
$Z_{e}^TM-MD_{\bf s}=FG^T$ 
 and if $s_i^n\neq e$ for $i=0,\dots,n-1$;


(c) $M=\sum_{j=1}^d\diag({\bf f}_j)C\diag({\bf g}_j)=\left(\frac{{\bf u}_i^T{\bf v}_j}{s_i-t_j}\right)_{i,j=0}^{n-1}$
if
$D_{\bf s}M-MD_{\bf t}=FG^T$. 
\end{theorem}
\begin{proof}
Parts $(t)$ and $(h)$ are taken from \cite[Examples 4.4.2 and 4.4.5]{P01}.
Part $(c)$ is taken from \cite[Example 1.4.1]{P01}.
To prove part ($v$), combine the equations
$D_{\bf s}M-MZ_{e}=FG^T$
and $Z_eZ_{1/e}^T=I$
(cf. (\ref{eqinv}))  and
deduce 
that 
$M-D_{\bf s}MZ_{1/e}^T=-F(Z_{1/e}G)^T$.
Then  obtain from
 \cite[Example 4.4.6 (part b)] {P01} 
that
$M=
e\diag(\frac{1}{s_i^n-e})_{i=0}^{n-1}\sum_{j=1}^d\diag({\bf f}_j)VZ_{1/e}(Z_{1/e}{\bf g}_j)^T$.
Substitute $eZ_{1/e}(Z_{1/e}{\bf g}_j)=Z_e(J{\bf g}_j)^T$ and obtain 
the claimed  expression of part ($v$).
Next  transpose the equation
$Z_{e}^TM-MD_{\bf t}=FG^T$
and yield
$D_{\bf s}M^T-M^TZ_{e}=-GF^T$.
 From 
 part $(v)$ 
obtain
$M^T=\diag(\frac{1}{e-s_i^n})_{i=0}^{n-1}\sum_{j=1}^d \diag({\bf g}_j)VZ_{e}(J{\bf f}_j)$.
Transpose this equation
and arrive at 
part $(v^T)$.
\end{proof}


\begin{table}
\caption{Operator matrices for the seven classes $\mathcal T$, 
$\mathcal H$,
$\mathcal V_{\bf s}$,
$\mathcal V^{-1}_{\bf s}$, 
$\mathcal V^{T}_{\bf s}$,
$\mathcal V^{-T}_{\bf s}$,
and
$\mathcal C_{\bf s,t}$} 
\label{taboperm}
\begin{center}
\renewcommand{\arraystretch}{1.2}
\begin{tabular}{|c|c|c|c|c|c|c|}
 \hline
 $\mathcal T$ & $\mathcal H$   & $\mathcal V_{\bf s}$ 
&$\mathcal V^{-1}_{\bf s}$
&$\mathcal V^T_{\bf s}$ 
& $\mathcal V^{-T}_{\bf s}$ &$\mathcal C_{\bf s,t}$\\ \hline 
$(Z_e,Z_f)$ & $(Z_e^T,Z_f)$  & $(D_{\bf s},Z_e)$ &  $(Z_e,D_{\bf s})$ &  
$(Z_e^T,D_{\bf s})$  &  $(D_{\bf s},Z_e^T)$&  $(D_{\bf s},D_{\bf t})$\\ \hline 
 $(Z_e^T,Z_f^T)$ & $(Z_e,Z_f^T)~$ & & & & &  \\ \hline
\end{tabular}
\end{center}
\end{table}

 By combining 
Theorems \ref{thcvdft}  and \ref{thdexpr}
we obtain the following results.

\begin{theorem}\label{thmbv1} 
Given a vector ${\bf v}$ of a dimension $n$
and a displacement generator of a length $d$
for a matrix $M$,
one can compute the product $M{\bf v}$
 by using $O(dn\log (n))$ flops for an $n\times n$
matrix $M$ 
in the classes $\mathcal T$ or 
$\mathcal H$
and by using $O(dn\log^2 (n))$ flops for an $n\times n$
matrix $M$ in $\mathcal V$, $\mathcal V^T$, or
$\mathcal C$.
\end{theorem}

\begin{remark}\label{reexpr}
By virtue of Theorem \ref{thdexpr}
the displacement operators $M\rightarrow AM-MB$
are nonsingular provided
 that $e\neq f$ in parts $(t)$ and $(h)$
and provided that 
$t_i^n\neq e$ for $i=0,\dots,n-1$ in parts $(v)$ and $(v^T)$.
We can  apply Theorem \ref{thzef} 
to satisfy these assumptions.
\end{remark}

\begin{remark}\label{reperm} (Cf. Theorem \ref{thcv1}.)
Parts $(v)$ and $(c)$ of Theorem  \ref{thdexpr} imply that
a row interchange preserves
the matrix structures of the Vandermonde and Cauchy types, 
whereas 
a  column interchange preserves
the matrix structures of the
transposed  Vandermonde and Cauchy types. 
\end{remark}

\section{Matrix operations in terms of displacement generators}\label{soper}

To accelerate pairwise multiplication  and the inversion of
structured matrices of large sizes 
 we express them as well as  the intermediate and final results 
of the computations
through short displacement generators
rather than the matrix entries. Such computations are
possible by virtue of the following simple results
from \cite{P00} and \cite[Section 1.5]{P01},
extending  \cite{P90}.

\begin{theorem}\label{thdgs}
Assume five matrices $A$, $B$, $C$, $M$ and $N$.
Then 
\begin{equation}\label{eqprm}
A(MN)-(MN)C=(AM-MB)N+M(BN-NC)~{\rm and}~
\end{equation}
\begin{equation}\label{eqinvm}
AM^{-1}-M^{-1}B=-M^{-1}(BM-MA)M^{-1}
\end{equation}
provided that the matrix multiplications and inversion involved
are well defined.
\end{theorem}

\begin{corollary}\label{codgdr}
For five matrices $A$, $B$, $F$, $G$, and $M$
of sizes $m\times m$, $n\times n$,
$m\times d$, $n\times d$, and $m\times m$, 
respectively,
let us write $F=F_{A,B}(M)$, $G=G_{A,B}(M)$,
 and  $d=d_{A,B}(M)$
if $AM-MB=FG^T$.
Then under the assumptions of Theorem \ref{thdgs}
 we obtain the following equations,
\begin{eqnarray*}
F_{A,B}(M^T)=-G_{B^T,A^T}(M^T),~G_{A,B}(M^T)=F_{B^T,A^T}(M^T), \\
F_{A,C}(MN)=(F_{A,B}(M)~|~M F_{B,C}(N)), \\
G_{A,C}(MN)=(N^TG_{A,B}(M)~|~G_{B,C}(N)),  \\
F_{A,B}(M^{-1})=-M^{-1}G_{B,A}(M),~G_{A,B}(M^{-1})=M^{-T}F_{B,A}(M),
\end{eqnarray*}

\noindent and so
$d_{A,B}(M^T)=d_{B^T,A^T}(M)$,
$d_{A,C}(MN)\le d_{A,B}(M)+d_{B,C}(N)$,  
$d_{A,B}(M^{-1})=d_{B,A}(M)$.
\end{corollary}

The corollary and Theorem \ref{thdexpr} 
together reduce the inversion of a nonsingular $n\times n$ matrix $M$ 
given with a displacement generator of a length $d$
to solving
$2d$ linear systems of equations with this coefficient matrix $M$,
rather than the $n$ linear systems $M{\bf x}_i={\bf e}_i$, $i=1,\dots,n$.

Given short displacement generators for 
the matrices $M$ and $N$,
we can apply Corollary \ref{codgdr}
and readily express short displacement generators for 
the matrices $M^T$
and $MN$
through the matrices $M$ and $N$ and their
 displacement generators, 
but the expressions 
for the displacement generator of
the inverse $M^{-1}$
involve the inverse  itself.

 


\section{Transformations of Displacement Matrix Structures}\label{sdtr}


Equations (\ref{eqvf}),
(\ref{eqfhrf}),
(\ref{eqvs})--(\ref{eqvs-t})
 link Cauchy and Van\-der\-monde matrix structures
together, by means of multiplication by  structured 
 matrices. We are going to generalize this technique.
 We begin with
 recalling some simple links
among Toeplitz, Hankel,
 and Van\-der\-monde matrices.
Then 
we will cover the approach
comprehensively.

\begin{theorem}\label{thcvtvh}
(i) $JH$ and $HJ$ are Toeplitz matrices if $H$
is a Hankel matrix, and vice versa. 
(ii) $V^TV=(\sum_{k=0}^{m-1}s_k^{i+j})_{i,j=0}^{n-1}$ is a Hankel matrix for
 any $m\times n$ Van\-der\-monde matrix $V=(s_i^{j})_{i,j=0}^{m-1,n-1}$.
\end{theorem}

\subsection{Maps and multipliers}\label{shld}

Recall that 
each of 
the five matrix classes
 $\mathcal T$, 
$\mathcal H$,
$\mathcal V$,
$\mathcal V^T$, 
and
$\mathcal C$
consists of the
 matrices $M$ 
whose displacement rank, $\rank(AM-MB)$ 
is small (in context) for 
a pair of operator matrices $(A,B)$ associated with this class
and representing its structure.
 Theorem \ref{thdgs} shows the impact of 
elementary matrix operations  on the associated operator matrices
$A$ and $B$. 
The operations of 
transposition and inversion 
 change the associated  pair $(A,B)$   
into $(-B^T,A^T)$ or $(-B,A)$.
If the inputs of the operations are in any of the classes   
$\mathcal T$, 
$\mathcal H$, and 
$\mathcal C_{\bf s,t}$, then so are the outputs.
Furthermore the transposition maps the classes 
$\mathcal V$ 
and $\mathcal V^T$  into one another, 
whereas inversion maps them
into the classes
$\mathcal V^{-1}$ 
and $\mathcal V^{-T}$,
respectively.
The  impact of multiplication on 
matrix structure is quite different.
By virtue of (\ref{eqprm}) and  Table \ref{taboper},
the map $M\rightarrow PMN$ can define 
the transition from the associated pair of operator matrices
 $(A,B)$
 to any new pair $(C,D)$ of our choice,
that is we can transform
the matrix structures of the five classes into each other  at will.
 The following theorem and Table \ref{tabmsmp}
specify such  transforms of the  structures
 given by the maps 
$M\rightarrow MN$,  $N\rightarrow MN$, and $M\rightarrow PMN$
for appropriate 
 multipliers $P$, $M$, and $N$. 

\begin{table}
\caption{Operator matrices for matrix product} 
\label{taboper}
\begin{center}
\renewcommand{\arraystretch}{1.2}
\begin{tabular}{|c|c|c|c|}
\hline 
$P$ & $M$ & $N$ & $PMN$ \\ \hline 
$C$ & $A$ & $B$ &  $C$ \\  \hline 
$A$ & $B$ & $D$  & $D$ \\ \hline 
\end{tabular}
\end{center}
\end{table}

\begin{theorem}\label{thmap} 
It holds that 

(i) $MN\in \mathcal T$ if the  pair of matrices $(M,N)$ is in any
of the  pairs of matrix classes
$(\mathcal T,\mathcal T)$,
$(\mathcal H,\mathcal H)$,
$(\mathcal V^{-1}_{\bf s},\mathcal V_{\bf s})$ and $(\mathcal V^{T}_{\bf s},\mathcal V^{-T}_{\bf s})$, 

(ii) $MN\in \mathcal H$ if the  pair $(M,N)$ is in any
of the pairs $(\mathcal T,\mathcal H)$,
$(\mathcal H,\mathcal T)$,
$(\mathcal V^{-1}_{\bf s},\mathcal V_{\bf s}^{-T})$ and $(\mathcal V^{T}_{\bf s},\mathcal V_{\bf s})$,

(iii) $MN\in \mathcal V_{\bf s}$ if
the  pair 
$(M,N)$ is in any
of the  pairs $(\mathcal V_{\bf s},\mathcal T)$,
$(\mathcal V^{-T}_{\bf s},\mathcal H)$, and $(\mathcal C_{\bf s,t},\mathcal V_{\bf t})$, 

(iv) $MN\in \mathcal V^T_{\bf s}$ if
the  pair 
$(M,N)$ is in any
of the  pairs  $(\mathcal T,\mathcal V^T_{\bf s})$,
$(\mathcal H,\mathcal V^{-1}_{\bf s})$ and $(\mathcal V^{T}_{\bf q},\mathcal C_{\bf q,s})$, 

(v) $MN\in \mathcal C_{\bf s,t}$ if
the  pair 
$(M,N)$ is in any
of the pairs  $(\mathcal C_{\bf s,q},\mathcal C_{\bf q,t})$,
$(\mathcal V^{-T}_{\bf s},\mathcal V^T_{\bf s})$ and $(\mathcal V_{\bf s},\mathcal V^{-1}_{\bf s})$,

(vii) $PMN\in \mathcal C_{\bf s,t}$ if
the  triple
$(M,N,P)$ is in any
of the  triples  $(\mathcal V_{\bf s},\mathcal H,\mathcal V^{T}_{\bf t})$ and 
$(\mathcal V^{-T}_{\bf s},\mathcal H,\mathcal V^{-1}_{\bf t})$. 
\end{theorem}

\begin{table}
\caption{Mapping matrix structures by means of multiplication} 
\label{tabmsmp}
\begin{center}
\renewcommand{\arraystretch}{1.2}
\begin{tabular}{|c|c|c|c|c|}
 \hline
$\mathcal T$ & $\mathcal H $ & $ \mathcal V_{\bf s}$& 
$\mathcal V^T_{\bf s}$ & $\mathcal C_{\bf s,t}$\\ \hline 
$\mathcal T\mathcal T$, $\mathcal V^{T}_{\bf s}\mathcal V^{-T}_{\bf s}$ & 
$\mathcal T\mathcal H$,  $\mathcal V^T_{\bf s}\mathcal V_{\bf s}$ & $\mathcal V_{\bf s}\mathcal T$, 
$\mathcal C_{\bf s,t}\mathcal V_{\bf t}$ &  $\mathcal T\mathcal V^T_{\bf s}$, 
$\mathcal H\mathcal V^{-1}_{\bf s}$  &  
$\mathcal V^{-T}_{\bf s}\mathcal V^T_{\bf t}$, 
$\mathcal V_{\bf s}\mathcal V^{-1}_{\bf t}$, $\mathcal C_{\bf s,q}\mathcal C_{\bf q,t}$\\  \hline 
~$\mathcal V^{-1}_{\bf s}\mathcal V_{\bf s}$, $\mathcal H\mathcal H$ & 
$\mathcal H\mathcal T$, $\mathcal V^{-1}_{\bf s}\mathcal V^{-T}_{\bf s}$  & 
$\mathcal V^{-T}_{\bf s}\mathcal H$ & 
$\mathcal V^T_{\bf q}\mathcal C_{\bf q,s}$ & $\mathcal V_{\bf s}\mathcal H\mathcal V^T_{\bf t}$, 
$\mathcal V^{-T}_{\bf s}\mathcal H\mathcal V^{-1}_{\bf t}$\\  \hline
\end{tabular}
\end{center}
\end{table}


The maps of Theorem \ref{thmap} and Table \ref{tabmsmp} 
hold for any choice of the multipliers $P$ and $N$
from the indicated  classes.
To simplify the computation
of
displacement generators for
 the products
$MN$, $MP$ and $MNP$,
 we can choose  
the multipliers  $J$, $V_{\bf r}$, $V_{\bf r}^T$, 
$V_{\bf r}^{-1}$, $V_{\bf r}^{-T}$,
 and
$C_{\bf p,r}$,
all
having displacement rank 1,
to represent the classes 
 $\mathcal H$, $\mathcal V_{\bf r}$,
 $\mathcal V_{\bf r}^T$,
$\mathcal V_{\bf r}^{-1}$, $\mathcal V_{\bf r}^{-T}$,
  and
$\mathcal C_{\bf p,r}$, respectively,
where ${\bf p}$ and ${\bf r}$ can stand for ${\bf q}$,  ${\bf s}$, and ${\bf t}$.
Hereafter we call this choice of multipliers {\em canonical}.
We call them {\em canonical and DFT-based} if 
up to the factor $J$
they are also DFT-based,
that is if ${\bf p}$ and ${\bf r}$ 
are of the form $f(\omega_n^{i})_{i=0}^{n-1}$. 
These multipliers are quasi\-unitary 
where $|f|=1$.
By combining 
Corollary \ref{codgdr} and
Theorem
 \ref{thmap}
we obtain the following
result.

\begin{corollary}\label{codcan} 
Suppose a
displacement generator of 
a length $d$ is given for
an $n\times n$ matrix $M$
of any of the classes 
$\mathcal T$, 
$\mathcal H$,
$\mathcal V$,
$\mathcal V^T$, and
$\mathcal C$. Then
$O(dn\log^2 (n))$ flops are sufficient to
 compute 
a displacement generator of 
a length at most $d+2$
for the matrix $PMN$
 of any other
of these classes where
 $P$ and $M$ are from the set of 
 canonical multipliers complemented by the
identity matrix.
The flop bound decreases to
$O(dn\log (n))$ where 
the canonical multipliers are DFT-based.
\end{corollary}

One can
 simplify the inversion 
of structured matrices $M$ of some important
classes
and 
the  solution of linear systems
$M{\bf x}={\bf u}$
by employing
 preprocessing
 $M\rightarrow PMN$
with appropriate
structured
 multipliers $P$ and $N$.



\subsection{The impact on displacements}\label{strmstr}


 
Theorems \ref{thhv} and \ref{thdiag}
of this subsection imply that in the canonical maps
of Theorem \ref{thmap} 
the displacement ranks grow
by at most 2 but possibly less
than that.
Our {\em constructive proofs 
of these theorems
 also
 specify the multipliers $P$ and $N$ and
 the displacement generators}  for the products
$PMN$ involved into the maps of Theorem \ref{thmap}.
In the maps supporting parts (a)--(e) of the following theorem 
we set   $P=I$ or $N=I$,
thus omitting one of the multipliers.
The theorem implicitly covers the maps where the matrices $M$ or $PMN$
belong to the classes $\mathcal V^T$, $\mathcal V^{-1}$, or 
$\mathcal V^{-T}$, because we can 
generate these maps by
transposing or inverting  the maps for 
$M\in \mathcal V$ and 
$PMN \in \mathcal V$.

\begin{theorem}\label{thhv}
Suppose a displacement generator
of a length $d$ 
is given
for a structured matrix $M$ 
of any of the four classes 
$\mathcal T$, 
$\mathcal H$,
$\mathcal V$, 
and
$\mathcal C$. Then
one can obtain a displacement generator of a length at most $d+2$
for 
a matrix $PMN$ belonging to any other of these classes by selecting
appropriate canonical multipliers 
$P$ and $N$
among the matrices $I$ (from the class $\mathcal T$), $J$
 (from the class $\mathcal H$), 
 Vandermonde  matrices $V$ and their transposes 
$V^T$.
Namely, if we assume canonical multipliers $P$ and $N$,
then we can compute a
displacement generator of the matrix $PMN$
having a length 
at most $d$ where the map $M\rightarrow PMN$
is between the matrices $M$ and $PMN$ in the classes
$\mathcal H$ and $\mathcal T$. This length
bound grows to
at most $d+2$ where            
$M$ is in the class $\mathcal T$ or $\mathcal H$,
whereas 
 $PMN \in\mathcal C$ or vice versa, 
where $M \in\mathcal C$
and $PMN$ 
is in the class $\mathcal T$ or $\mathcal H$.
Displacement generators have lengths  at most
  $d+1$ in the maps
$M\rightarrow PMN$ for 
all other transitions among the classes
$\mathcal T$, $\mathcal H$,
$\mathcal V$ and $\mathcal C$.               
\end{theorem}
\begin{proof}
We specify some maps $M\rightarrow MNP$ 
that support the claims of the theorem. 
One can 
vary and combine these maps
as well as the other maps of Theorem \ref{thmap}  and Table \ref{tabmsmp}. 

(a) $\mathcal T\rightarrow  \mathcal H$, $PMN=JM$.
Assume
a matrix $M\in \mathcal T$,
a pair of distinct 
scalars $e$ and $f$,
and a pair of $n\times d$ matrices $F=F_{Z_e,Z_f}(M)$ 
and $G=G_{Z_e,Z_f}(M)$ for
$d=d_{Z_e,Z_f}(M)$
satisfying the
displacement equation $Z_eM-MZ_{f}=FG^T$
(cf. Theorem \ref{thdexpr}). 
Pre-multiply this  equation by the matrix $J$ to obtain
 $JZ_eM-(JM)Z_{f}=JFG^T$. 
 Rewrite the term
$JZ_eM=JZ_eJJM$ as $Z_e^TJM$
by observing that
$JZ_eJ=Z_e^T$ (cf.
(\ref{eqrever}) for $f=e$).
Obtain
$Z_e^T(JM)-(JM)Z_{f}=JFG^T$.
Consequently
$F_{Z_e^T,Z_f}(JM)=JF$,
$G_{Z_e^T,Z_f}(JM)=G$,
 $d_{Z_e^T,Z_{f}}(JM)=d_{Z_e,Z_{f}}(M)$, 
and  $JM\in \mathcal H$.

(b) $\mathcal T\rightarrow  \mathcal V$, $PMN=VM$.
Keep the assumptions 
 of part (a) and
fix 
 $n$ scalars  $s_0,\dots,s_{n-1}$.
Pre-multiply the displacement equation 
$Z_eM-MZ_{f}=FG^T$
by 
the Vandermonde matrix $V=(s_i^{j})_{i,j=0}^{n-1}$
to
obtain $VZ_eM-(VM)Z_{f}=VFG^T$.
Write
${\bf s}=(s_i)_{i=0}^{n-1}$
and
substitute 
equation 
(\ref{eqvand})
 to yield
$D_{\bf s}(VM) -(VM)Z_{f}=VFG^T+(s_i^n-e)_{i=0}^{n-1}{\bf e}_n^TM=
F_{VM}G_{VM}^T$ for $F_{VM}=(VF~|~(s_i^n-e)_{i=0}^{n-1})$ and
$G_{VM}^T=\begin{pmatrix}
G^T \\
{\bf e}_n^TM
\end{pmatrix}$. 
So 
 $ d_{D_{\bf s},Z_{f}}(VM) \le d_{Z_e,Z_{f}}(M)+1$
and $VM\in \mathcal V_{\bf s}$. 

(c) $\mathcal H\rightarrow  \mathcal T$,
$PMN=MJ$.
Assume
a matrix $M\in \mathcal H$,
a pair of  
scalars $e$ and $f$,
and a pair of $n\times d$ matrices $F$ and $G$ for
$d=d_{Z_e,Z_f^T}(M)$
satisfying
the displacement equation $Z_eM-MZ_{f}^T=FG^T$
(cf. Theorem \ref{thdexpr}).
Post-multiply it by the matrix $J$ to obtain
 $Z_e(MJ)-MZ_{f}^TJ=FG^TJ$.
Express the term
$MZ_f^TJ=MJJZ_f^TJ$ as $MJZ_f$ 
(cf. (\ref{eqrever}))
to obtain 
$Z_e(MJ)-(MJ)Z_{f}=FG^TJ=F(JG)^T$
and consequently
$F_{Z_e,Z_f}(JM)=F$,
$G_{Z_e,Z_f}(JM)=JG$,
 $d_{Z_e,Z_{f}}(MJ)=d_{Z_e,Z_{f}^T}(M)$ 
and $MJ\in \mathcal T$.

(d) $\mathcal H\rightarrow  \mathcal V$.
Compose the maps of 
parts (c) and (b).

(e) $\mathcal V\rightarrow  \mathcal H$,
 $PMN=V^TM$.
Assume
$n+2$  scalars
$e,f,
s_0,\dots,s_{n-1}$,
a matrix $M\in \mathcal V$, 
and its displacement generator
given by
$n\times d$ matrices $F$ and $G$ such that
$D_{\bf s}M-MZ_f=FG^T$
(cf. (\ref{eqvand})).
Pre-multiply this equation by the  transposed Vandermonde 
matrix $V^T=(s_j^{i})_{i,j=0}^{n-1}$ to obtain
$V^TD_{\bf s}M-(V^TM)Z_f=V^TFG^T$
for
${\bf s}=(s_i)_{i=0}^{n-1}$.
Apply equation (\ref{eqvandt})
to express  the matrix $V^TD_{\bf s}$
 and obtain
$Z_e^T(V^TM)-(V^TM)Z_f=V^TFG^T+{\bf e}_n((s_i^n-e)_{i=0}^{n-1})^TM=
F_{V^TM}G_{V^TM}^T$
for $F_{V^TM}=(V^TF~|~{\bf e}_n)$
and $G_{V^TM}^T=\begin{pmatrix}
G^T \\
((s_i^n-e)_{i=0}^{n-1})^TM
\end{pmatrix}$.
So 
 $d_{Z_e^T,Z_{f}}(V^TM)\le d_{D_{\bf s},Z_f}(M)+1$
and  
$V^TM\in \mathcal H$. 

(f) $\mathcal V\rightarrow  \mathcal T$.
Compose  the maps of 
parts (e) and (c).

(g) $\mathcal V\rightarrow  \mathcal C$,
$PMN=MJV^T$.
Assume  
$2n+1$  scalars $e$,
$s_0,\dots,s_{n-1},t_0,\dots,t_{n-1}$,
a matrix $M\in \mathcal V$, 
and its displacement generator
given by
$n\times d$ matrices $F$ and $G$.
Post-multiply the equation 
$D_{\bf s}M-MZ_e=FG^T$
(cf. (\ref{eqvand}))
by the matrix $JV^T$
where  $V^T=(t_j^{i})_{i,j=0}^{n-1}$
is the transposed Vandermonde 
matrix, substitute $Z_eJ=JZ_e^T$, and obtain
$D_{\bf s}(MJV^T)-MJZ_e^TV^T=FG^TJV^T$
for
 ${\bf s}=(s_i)_{i=0}^{n-1}$.
Apply equation (\ref{eqvandt}) to
express the matrix $Z_e^TV^T$ and
 obtain
$D_{\bf s}(MJV^T)-(MJV^T)D_{\bf t}=FG^TJV^T-MJ{\bf e}_n((t_i^n-e)_{i=0}^{n-1})^T=
F_{MJV^T}G_{MJV^T}^T$
where $F_{MJV^T}=(F~|~MJ{\bf e}_n)$
and $G_{MJV^T}^T=\begin{pmatrix}
G^TJV^T \\
((e-t_i^n)_{i=0}^{n-1})^T
\end{pmatrix}.$
So 
 $d_{D_{\bf s},D_{\bf t}}(MJV^T)\le d_{D_{\bf s},Z_e}(M)+1$
and  
$MJV^T\in \mathcal C_{\bf s,t}$. 

We can alternatively write $PMN=MV^{-1}$
for $V=V_{\bf t}$. (The matrix $V$ can be readily inverted 
where it is FFT-based, that is where $V=V_f$.)
Post-multiply the equation $D_{\bf s}M-MZ_e=FG^T$
(cf. (\ref{eqvand}))
  by the matrix $V^{-1}=V_{\bf t}^{-1}$,
for ${\bf t}=(t_i)_{i=0}^{n-1}$,
to obtain $D_{\bf s}MV^{-1}-MZ_eV^{-1}=FG^TV^{-1}$.
Pre- and post-multiply
by $V^{-1}$
equation
(\ref{eqvand})  for ${\bf s}$ replaced by  ${\bf t}$ and
obtain $Z_eV^{-1}=V^{-1}D_{\bf t}-V^{-1}(t_i^n-e)_{i=0}^{n-1}{\bf e}_n^TV^{-1}$.
Substitute this expression for $Z_eV^{-1}$ 
into the above equation
 and obtain
$D_{\bf s}(MV^{-1})-(MV^{-1})D_{\bf t}=FG^TV^{-1}-V^{-1}(e-t_i^n)_{i=0}^{n-1}{\bf e}_n^TV^{-1}=F_{MV^{-1}}G_{MV^{-1}}^T$
for $F_{MV^{-1}}=(F~|~V^{-1}(e-t_i^n)_{i=0}^{n-1})$
and $G_{V^{-1}M}^T=\begin{pmatrix}
G^TV^{-1} \\
{\bf e}_n^TV^{-1}
\end{pmatrix}$.
So 
 $d_{D_{\bf s},D_{\bf t}}(VM)\le d_{D_{\bf s},Z_e}(M)+1$
and  
$MV^{-1}\in \mathcal C_{\bf s,t}$.

(h) $\mathcal C\rightarrow  \mathcal V$,
$PMN=MV$.
Assume 
$2n+1$ scalars $e$,
$s_0,\dots,s_{n-1},t_0,\dots,t_{n-1}$,
a matrix $M\in \mathcal C$, 
and its displacement generator
given by
$n\times d$ matrices $F$ and $G$ such that
$D_{\bf s}M-MD_{\bf t}=FG^T$
for ${\bf s}=(s_i)_{i=0}^{n-1}$
and ${\bf t}=(t_i)_{i=0}^{n-1}$
(cf. (\ref{eqcch})).
Post-multiply this equation by the  Vandermonde 
matrix $V=(t_i^{j-1})_{i,j=0}^{n-1}$ to obtain
$D_{\bf s}(MV)-MD_{\bf t}V=FG^TV$.
Express the matrix $D_{\bf t}V$
from matrix equation (\ref{eqvand})
 and obtain
$D_{\bf s}(MV)-(MV)Z_e=FG^TV+M(t_i^n-e)_{i=0}^{n-1}{\bf e}_n^T=
F_{MV}G_{MV}^T$
where $F_{MV}=(F~|~M(t_i^n-e)_{i=0}^{n-1})$
and $G_{MV}^T=\begin{pmatrix}
G^T V\\
{\bf e}_n^T
\end{pmatrix}.$
So 
 $d_{D_{\bf s},Z_e}(MV)\le d_{D_{\bf s},D_{\bf t}}(M)+1$
and  
$MV\in \mathcal V_{\bf s}$. 

(i) $\mathcal C\rightarrow  \mathcal T$.
Compose  the maps of parts (h) and (f).

(j) $\mathcal C\rightarrow  \mathcal H$.
Compose  the maps of parts (h) and (e).

(k) $\mathcal T\rightarrow  \mathcal C$.
Compose  the maps of parts (b) and (g).

(i) $\mathcal H\rightarrow  \mathcal C$.
Compose  the maps of parts (d) and (g).
\end{proof}                   


Multiplications by a Cauchy matrix
keeps a matrix in any of the classes
$\mathcal T$, $\mathcal H$, $\mathcal V$
and $\mathcal C$, but changes 
a diagonal operator matrix. 
Next we specify  
the impact on the displacement.

\begin{theorem}\label{thdiag}
Assume $2n$ distinct scalars 
$s_0,\dots,s_{n-1},t_0,\dots,t_{n-1}$,
defining two vectors
${\bf s}=(s_i)_{i=0}^{n-1}$ and ${\bf t}=(t_j)_{j=0}^{n-1}$ and
a nonsingular Cauchy
matrix $C=C_{\bf s, t}=(\frac{1}{s_i-t_j})_{i,j=0}^{n-1}n$
(cf. part (i) of Theorem \ref{thcvdt}).
Then for any pair of operator matrices $A$ and $B$ we have

(i) $d_{A,D_{\bf t}}(MC)\le d_{A,D_{\bf s}}(M)+ 1$ and
  
(ii) $d_{D_{\bf s},B}(CM)\le d_{D_{\bf t},B}(M)+ 1$.
\end{theorem}
\begin{proof}        
(i) We have $d_{A,D_{\bf s}}(M)=\rank (AM-MD_{\bf s})
=\rank(AMC-MD_{\bf s}C)$.
Furthermore
$AMC-MCD_{\bf t}=AMC-MD_{\bf s}C+MD_{\bf s}C-MCD_{\bf t}
=(AM-MD_{\bf s})C+M(D_{\bf s}C-CD_{\bf t})$.
Substitute   equation (\ref{eqcch}) and deduce that 
$AMC-MCD_{\bf t}=(AM-MD_{\bf s})C+M{\bf e}{\bf e}^T$.
Therefore $d_{A,D_{\bf t}}(MC)=\rank (AMC-MCD_{\bf t})\le 
\rank ((AM-MD_{\bf s})C)+1=\rank (AM-MD_{\bf s})+1
=d_{A,D_{\bf s}}(M)+1$.

(ii) We have 
$D_{\bf s}CM-CMB=D_{\bf s}CM-CD_{\bf t}M+CD_{\bf t}M-CMB=
(D_{\bf s}C-CD_{\bf t})M+
C(D_{\bf t}M-MB)$.
Substitute   equation (\ref{eqcch}) and deduce that
$D_{\bf s}CM-CMB=C(D_{\bf t}M-MB)+{\bf e}{\bf e}^TM$.
Therefore $d_{D_{\bf s},B}(CM)= \rank (D_{\bf s}CM-CMB)\le
 \rank (C(D_{\bf t}M-MB))+ 1=
\rank (D_{\bf t}M-MB)+ 1=
d_{D_{\bf t},B}(M)+1$.
\end{proof}                 


\subsection{Canonical and DFT-based transformations of the matrices of the classes $\mathcal T$,  $\mathcal H$,
$\mathcal V$ and $\mathcal V^T$
into CV-like  matrices}\label{scv}


By combining  equations
(\ref{eqvand}), (\ref{eqvandt}), and
(\ref{eqprm})
one can deduce that
multiplication by a Vandermonde multiplier 
$V=(s_i^{j-1})_{i,j=0}^{n-1}$ or
by its transpose 
increases 
the length of a displacement generator 
by at most $1$, but
equations (\ref{eqvand}) and (\ref{eqvandt}) imply that
such  multiplication
 does not increase the length at all
where $s_i^n=e$ for $i=0,\dots,n-1$ 
and for a scalar $e$, employed 
in the operator matrices $Z_e$ and $Z_e^T$
of 
the Van\-der\-monde displacement map.
This suggests choosing the vectors 
${\bf s}=(e\omega_n^{i-1})_{i=0}^{n-1}$
and ${\bf t}=(f\omega_n^{i-1})_{i=0}^{n-1}$
and employing the DFT-based multipliers $V_e$
and $V_f$
of (\ref{eqvf}),
in particular
in our maps
supporting part (g)  of Theorem \ref{thhv}. 
Then the output
 matrices of the class $\mathcal {CV}$
would have the same displacement ranks as the 
input matrices $M$.
Furthermore the inverse of the matrix  $V=V_{\bf t}$,
 employed in our second map supporting part (g), 
would turn into DFT-based matrix $V_f$,
and we could invert it
and multiply it by a vector
by using $O(n\log (n))$ flops
(cf. Theorem \ref{thcvdft}).
We deduce the following results by reexamining the proof of 
Theorem \ref{thhv} and applying transposition.

\begin{theorem}\label{thfmap}
The canonical DFT-based
multipliers 
from the proof of
 Theorem \ref{thhv}
for the basic vectors 
${\bf s}=(e\omega_n^{i})_{i=0}^{n-1}$
and ${\bf t}=(f\omega_n^{i})_{i=0}^{n-1}$
and for some appropriate complex scalars 
$e$ and $f$
 support the following transformations of matrix classes
(in both directions),
 $\mathcal T\leftrightarrow \mathcal {FV}\leftrightarrow \mathcal {FCF}$,
 $\mathcal H\leftrightarrow \mathcal {FV}\leftrightarrow \mathcal {FCF}$,
 $\mathcal V\leftrightarrow \mathcal {CF}$, 
 $\mathcal V^T\leftrightarrow \mathcal {FC}$, and
$\mathcal V\cup \mathcal V^T\leftrightarrow \mathcal {CV}$.
The multipliers
are quasi\-unitary 
(and thus the transformations are numericaly stable)
where $|e|=|f|=1$. 
\end{theorem}

By combining our second map in the proof of part (g)
of  Theorem \ref{thhv} with our map from its part (b)
and choosing 
${\bf t}=(f\omega_n^{i})_{i=0}^{n-1}$,
 so that the $2n$ knots $s_0,t_0,\dots,s_{n-1},t_{n-1}$
are equally spaced on the unit circle $\{z:~|z|=1\}$,
we can obtain  canonical DFT-based transforms 
$\mathcal T\rightarrow  \mathcal C=\Omega\mathcal T\diag(f^{i})_{i=0}^{n-1}\Omega^H$, 
which are quasi\-unitary where $|f|=1$.
For $f=\omega_{2n}$ they turn into the celebrated map
employed in the papers \cite{H95},
\cite{GKO95},
\cite{G98},
\cite{MRT05},
\cite{R06},
\cite{CGS07},
\cite{XXG12}.
The following theorem shows the implied map of
the displacement generators
(see the proofs of parts (b) and (g) of Theorem \ref{thhv},
using the second map  supporting part (g)).

\begin{theorem}\label{thtc}
Suppose $Z_1M-MZ_{-1}=FG^T$
for an $n\times n$ matrix $M$ and 
$n\times d$ matrices $F$  and $G$
and write $P=\Omega$, $N=D_0^H\Omega^H$,
$C=PMN$, $D_0=\diag(\omega_{2n}^{i})_{i=0}^{n-1}$, and $D=D_0^2=\diag(\omega_{n}^{i})_{i=0}^{n-1}$. Then 
$DC-\omega_{2n}CD=F_CG_C^T$ for
$F_C=\Omega F$ and $G_C=\Omega D_0G$.
\end{theorem}

The theorem and the supporting 
 canonical DFT-based map 
$\mathcal T\rightarrow \mathcal C$
are a special case of Theorem \ref{thhv} 
and its transforms of matrix structures, extending \cite{P90}.
In his letter of 1991,
 reproduced in \cite[Appendix C]{P11}, G. Heinig has   
acknowledged studying the  paper \cite{P90}, although
in \cite{H95} he deduced Theorem \ref{thtc}
from Theorem \ref{thcpw} rather than supplying more general results 
based on Theorem \ref{thhv}.


\section{HSS matrices}\label{shss}


In the next four sections we study HSS 
matrices, employ them to  approximate Cauchy matrices, and
combine these results with the displacement and functional transformations
of matrix structures to devise more efficient algorithms.

\begin{definition}\label{defqs}
As in Section \ref{s1}, ``{\em HSS}" stands for ``hierarchically semiseparable".
An $n\times n$ matrix is $(l,u)$-HSS 
if its diagonal blocks consist of $O((l+u)n)$
entries, if $l$  is the maximum rank of all its subdiagonal blocks, 
and if $u$  is the maximum rank of all its superdiagonal blocks,
that is blocks of all sizes lying strictly below or strictly above the block 
 diagonal, respectively.
\end{definition}
 
HSS matrices extend the class of banded matrices and their inverses,
and similar extensions are known under the names of
 matrices with a low Hankel rank, quasi\-se\-parable, weakly,
recursively or 
sequentially semiseparable matrices, 
and rank structured matrices. See 
\cite{B10}, \cite{CDGLP06},
\cite{CGS07},
\cite{DV98},
\cite{EGH13a}, \cite{EGH13b},
 \cite{GKK85}, \cite{T00}, 
 \cite{VVGM05}, 
\cite{VVM07}, \cite{VVM08},  \cite{X13},
  \cite{XXG12}, \cite{XXCB},
  and the bibliography therein on the long
history of the study of these matrix classes
and see   \cite{B10}, \cite{BY13}, \cite{DGR96}, 
 \cite{CGR98}, \cite{GR87}, \cite{LRT79},  \cite{P93}, \cite{PR93},
and the bibliography therein
 on the related subjects of 
FMM, Matrix Compression, and Nested Dissection  algorithms.

One can readily express
the $n^2$ entries 
of  an $(l,u)$-HSS matrix 
of size $n\times n$  
 via $O((l+u)n)$ parameters
of a generalized
generator and can
multiply this matrix  
by a vector by using  $O((l+u)n)$ flops.
If the matrix is nonsingular,
then its inverse is also an $(l,u)$-HSS
matrix, and 
$O((l+u)^3n)$ flops are sufficient to compute
 a generalized generator expressing it
via $O((l+u)n)$ parameters. Having computed such a generator,  
one can solve a linear system with this matrix by using
$O((l+u)n)$ additional flops.
See  \cite{DV98}, \cite{EG02}, \cite{MRT05}, \cite{CGS07}, \cite{XXG12}, \cite{XXCB},
\cite{Pa}, and 
the references therein
on 
 supporting algorithms and their
efficient implementation.
Our next goal is the design of fast approximation algorithms for CV matrices 
by means of their approximation by slightly
generalized HSS matrices, to which we extend
fast HSS algorithms.


\section{Low-rank approximation of certain Cauchy matrices}\label{snrqs}


\begin{definition}\label{defss} (See  \cite[page 1254]{CGS07}.)
A pair of complex points $s$ and $t$
is $(\theta,c)$-{\em separated} 
for $0\le \theta<1$
and a complex point $c$ (a center)
if $|\frac{t-c}{s-c}|\le \theta$.
Two sets  of complex numbers  $\mathbb S$ and 
$\mathbb T$ are $(\theta,c)$-{\em separated} from one another
if every pair of elements $s\in \mathbb S$
and $t\in \mathbb T$ 
is $(\theta,c)$-separated from one another
for the same pair  $(\theta,c)$.
\end{definition}


\begin{theorem}\label{thss1}  (Cf. \cite{MRT05}, \cite[Section 2.2]{CGS07}.)
Suppose  $C=(\frac{1}{s_i-t_j})_{i,j=0}^{n-1}$
is a Cauchy matrix defined by two sets 
of parameters
$\mathbb S=\{s_0,\dots,s_{n-1}\}$ and 
$\mathbb T=\{t_0,\dots,t_{n-1}\}$. 
 Suppose these sets are $(\theta,c)$-{\em separated} from one another
for $0<\theta<1$ and a center $c$
and write 
\begin{equation}\label{eqdlt}
\delta=\delta_{c,\mathbb S}=\min_{i=0}^{n-1} |s_i-c|.
\end{equation}
 Then for every positive integer $k$ 
it is sufficient to use $2kn+2n-2$ flops to
 compute
 two matrices 
$F=(1/(s_i-c)^{h+1})_{i,h=0}^{n-1,k},~G^T=((t_j-c)^h)_{j,h=0}^{n,k}$
that support the representation of the matrix 
$C$ as $C=\widehat C+E$ where 
$\widehat C=FG^T,~\rank (\widehat C)\le k+1$,
$|E|\le \frac{\theta^k}{(1-\theta)\delta}~{\rm for~all~pairs}~ \{i,j\}$.
\end{theorem}


\begin{remark}\label{redl}
 We can replace $\delta=\delta_{c,\mathbb S}=\min_{i=0}^{n-1} |s_i-c|$
by $\delta=\delta_{c,\mathbb T}=\min_{j=0}^{n-1} |t_j-c|$
because 
$C_{\bf s,t}^T=-C_{\bf t,s}$ (cf. (\ref{eqctr})).  
\end{remark}

\begin{remark}\label{renmr}
Unless the values $1-\theta>0$ and  
$\delta$ of (\ref{eqdlt})
are small,
the upper bound of Theorem \ref{thss1} on the norm $|E|$ is small 
already for moderately large integers $k$. Then 
Theorem \ref{thss1} 
implies an upper bound $k+1$ on the numerical rank of 
the large subclass of
 Cauchy  matrices $C=(\frac{1}{s_i-t_j})_{i,j=0}^{n-1}$
whose
parameter sets $\mathbb S=\{s_0,\dots,s_{n-1}\}$ and 
$\mathbb T=\{t_0,\dots,t_{n-1}\}$ 
are $(\theta,c)$-separated from one another
for an appropriate center $c$.
Even if this property 
holds just for 
a subset of the set $\mathbb S$ that defines an $l \times n$ Cauchy submatrix
where $l>k+1$, then this submatrix and consequently
the matrix $C$ as well are ill conditioned.
In particular since all knots $t_0,\dots,t_{n-1}$ of a CV matrix
lie on the unit circle $\{z:~|z|=1\}$, they
are $(\theta,0)$-separated (with $\theta$ not close to 1)
from every knot $s_i$ not lying close to this circle,
and so a CV matrix is ill conditioned unless all 
but at most $l$
of
its knots $s_i$ lie on or near this circle.
\end{remark}


\section{Local low-rank approximation of CV   
matrices}\label{shssapr}


Theorem \ref{thss1} defines a low-rank approximation of a CV matrix
where its two knot sets are separated
by a {\em global center} $c$.
Generally a CV matrix has no such center, but next we show that it always has
a set of {\em local centers} that support approximation by generalized HSS matrices.
We begin with a simple lemma that expresses the distances between the 
two points of the unit circle $\{z:~|z|=1\}$ and from its point 
 to a sector.
\begin{lemma}\label{learc}
Suppose
$0\le \phi<\phi'< \phi''\le 2\pi$, $\phi'-\phi\le \phi-\phi'' +2\pi$,
$\tau=\exp(\phi\sqrt{-1})$, $\tau'=\exp(\phi'\sqrt{-1})$,  and  $\tau''=\exp(\phi''\sqrt{-1})$
and let
$\Gamma(\phi',\phi'')=
\{r\exp(\mu\sqrt {-1}):~r\ge 0,~0\le \phi'\le \mu<\phi''\le 2\pi\}$
denote the 
semi-open sector on the complex plane bounded by two rays from the origin passing through
the points $\tau'$ and $\tau''$.
Then (i) $|\tau'-\tau|=2\sin ((\phi'-\phi)/2)$
and (ii) the distance from the point $\tau$ to the sector $\Gamma(\phi',\phi'')$
is equal to $\sin(\phi'-\phi)$.
\end{lemma}

\begin{theorem}\label{thlrba}
Assume sufficiently large positive integers $k$, $h$ and $n$, 
a complex scalar $e$,  
and a CV matrix $C=C_{{\bf s},e}=(\frac{1}{s_i-t_j})_{i,j=0}^{n-1}$ 
such that 
$t_j=e\omega_n^{j}$ for $j=0,\dots,n-1$,
$kh=n$, the integers $k$ and $h$ are not small,  and $e\ne 0$. 
Then there is an  $n\times n$ 
permutation  matrix $P$ such that
 $PC$ is a block vector $PC=(C_0,\dots,C_{k-1})$
where
a {\em basic block column} $C_p$ has size $n\times h$.
 Furthermore consider the  first row of the
matrix $C$ adjacent to its  last row
(as if they were glued together).
Then 
 every basic block column $C_p$ can be partitioned
 into an $\widehat n_p\times h$ {\em extended diagonal block}  
$\widehat \Sigma_p$ and an 
$(\widehat n-n_p)\times h$
{\em admissible block} $\widehat N_p$ (see Remark \ref{reglu}),
such that the blocks $\widehat \Sigma_0,\dots,\widehat \Sigma_{k-1}$ 
have $3hn$ entries overall, whereas every admissible block
 $N_p=(\frac{1}{s_i-t_j})_{i\in \widehat{\mathbb S}_q,j\in \mathbb T_q}$ is
associated  with a pair of knot sets $\widehat{\mathbb S}_q\subseteq \mathbb S$
and $ \mathbb T_q\subset \mathbb T$
that are $(\theta, c_p)$-separated from one another for a center $c_p=\exp(\psi_p\sqrt {-1})$
and $\theta=2\sin (\mu)/\sin (\nu)\approx \tilde \theta= 2\mu/\nu$  where
 $2\mu=\max\{|c_p-\omega_n^{ph}|,~|c_p-\omega_n^{(p+1)h}|\}$,
$\nu=\min\{|2 (p-1)\pi k-\psi_p|,~|2(p+2)\pi k-\psi_p|\}$,
 and $\psi_p$ is any number
satisfying
$2p \pi/k \le \psi_p< 2(p+1)\pi/k$.
In particular $\mu=0.5\pi/k$, $\nu=3\pi/k$, and $\tilde \theta= 1/3$
 provided $c_p$ is the midpoint of the arc $\mathbb A_p$ of the unit circle $\{z:~|z|=1\}$
with the endpoints $\omega_n^{ph}$ and $\omega_n^{(p+1)h}$, that is
 provided
$\psi_p=(2p+1)\pi/k$, whereas 
$\mu\le 0.75 \pi/k$, $\nu\ge 2.5\pi/k$, and $\tilde \theta= 3/5$
 provided $c_p$ is a point on the arc $\mathbb A'_p$ with the end points
$\omega_{2n}^{(2p+0.5)h}$ and $\omega_{2n}^{(2p+1.5)h}$, that is  provided 
$(2p+0.5)\pi/k\le\psi_p\le(2p+1.5)\pi/k$.
\end{theorem}

\begin{proof} 
With no loss of generality assume that $e=1$ because the claimed properties are
readily extended from the matrix $C_{{\bf s},1}$ to the matrix 
$\frac{1}{e}C_{{\bf s},1}=C_{e{\bf s},e}$.
Represent the knots $s_0,\dots,s_{n-1}$ of the set $\mathbb S$ in polar coordinates,
$s_i=r_i\exp(2\pi\phi_i\sqrt {-1})$ where 
$r_i\ge 0$, $0\le\phi_i<2\pi$, $\phi_i=0$ if $r_i=0$,
and $i=0,1,\dots,n-1$. Re-enumerate all values $\phi_i$
to have them in nondecreasing order and
 to have $\phi_0^{(\rm new)}=\min_{i=0}^n\phi_i$
and let $P$ denote the permutation matrix that defines this re-enumeration.
To simplify our notation assume 
 that already 
the original enumeration has these  properties, that is $P=I$.
Let $\mathbb S_p$ and $\mathbb T_p$ 
 denote the  two subsets of the sets $\mathbb S$ and $\mathbb T$, respectively,
that lie in the semi-open sector  of the complex plane
$\Gamma_p=\{z=r\exp(\psi\sqrt{-1}):~r\ge 0, ~
2\pi p/k\le \psi<2\pi (p+1)/k \}$, 
bounded by the pair of the rays 
 from
 the origin to
the points $\omega_n^{ph}$ 
and $\omega_n^{(p+1)h}$.
Define the block partition 
$C=(C_{p,q})_{p,q=0}^{k-1}$ with the blocks $C_{p,q}=(\frac{1}{s_i-t_j})_{i\in \mathbb S_p,j\in \mathbb T_q}$
for $p,q=0,\dots,k-1$. Then declare the first block row
of the matrix $(C_{p,q})_{p,q=0}^{k-1}$ adjacent to its last row,
that is declare the blocks $C_{0,q}$ and  $C_{k-1,q}$ pairwise adjacent for every $q$
(as if the two rows were glued together),
and 
 partition every block column $C_{.,q}=(C_{p,q})_{q=0}^{k-1}$ into
the extended 
diagonal block $\widehat\Sigma_q=
(\frac{1}{s_i-t_j})_{i\in{\mathbb S}'_{q},j\in \mathbb T_q}$
and the admissible block 
$\widehat N_{q}=
(\frac{1}{s_i-t_j})_{i\in \widehat{\mathbb S}_q,j\in \mathbb T_q}$, where
$\widehat{\mathbb S}_q=\mathbb S-{\mathbb S}'_{q}$, 
${\mathbb S}'_{q}=\mathbb S_{q-1\mod k}\cup\mathbb S_q\cup\mathbb S_{q+1\mod k}$,
 $q=0,\dots,k-1$.
Clearly the $k$ diagonal blocks $\Sigma_p=C_{p,p}$ of sizes
$n_p\times h$ for $p=0,\dots,k-1$
have $h\sum_{p=0}^{k-1}n_p=hn$ entries overall, and this overall number is tripled in the
extension to the blocks $\widehat \Sigma_p$ because $\widehat n_p=n_{p-1\mod k}+n_q+n_{p+1\mod k}$,
and so $h\sum_{p=0}^{k-1}\widehat n_p=3h\sum_{p=0}^{k-1}n_p=3hn$.
Furthermore Lemma \ref{learc} implies that the sets $\widehat{\mathbb S}_q$ and
$\mathbb T_q$ defining the admissible block
$\widehat N_{q}$
are $(\theta,c_q)$-separated from one another
for  $\theta$ and $c_p$ defined in the theorem.
 \end{proof}

\begin{remark}\label{reglu}
Every block $\widehat \Sigma_p$ and  $\widehat N_p$ is made up 
of a pair 
blocks of the matrix $CP$ that 
are either adjacent to one another 
 or become  adjacent
if we declare that the first
row of the matrix is  adjacent
to its last row.
\end{remark}

\begin{remark}\label{regdn}
Even if $k$
does not divide $n$ we can still partition the unit circle
by $k$ equally spaced points, then
partition the complex 
plane into $k$ sectors 
 accordingly, represent the matrix $C$ as a $k\times k$
block matrix $(C_{p,q})_{p,q=0}^{k-1}$,  and define the basic block columns and
the
diagonal,  extended diagonal, and admissible blocks.
The only  change in the claims and proofs is that we would allow
 the number of columns of the latter blocks to vary slightly,
by at most 1, as $q$ varies. The techniques of
the adaptive FMM \cite{CGR98}
enable us to handle the case of any distribution of the knots $s_i$ on 
a circle as well as on a line or 
smooth curve of a bounded length. See Section \ref{svmv} on further extensions. 
\end{remark}

\begin{remark}\label{remerg} {\bf (Recursive merging.)}
Our analysis and  
results hold for any
positive integer $k$. We  
recursively apply them by
 partitioning the unit circle by arcs
whose lengths increase 
at every recursive step. More precisely, at every recursive step 
we merge a pair of the adjacent arcs of the current finer partition
of the circle
into a single arc of the new coarser  partition.
Then we redefine the diagonal,
extended diagonal and admissible blocks.
This {\em recursive merging}
 dramatically enhances the power of Theorem \ref{thlrba}
for supporting fast CV algorithms. 
\end{remark}

\begin{remark}\label{redelt} {\bf (Bounding the distance from the centers to the knot set} $\mathbb S${\bf .)}
There are exactly $n$ elements $s_0,\dots,s_{n-1}$ in the set 
$\mathbb S$. Therefore for every $p$
we can choose a center $c_p$ on the arc $\mathbb A'_p$
at the distance at least $2\sin(\pi/(8kn))$
from this set and thus obtain the bound 
\begin{equation}\label{eqdelt}
\delta\ge \delta_-=2\sin(\pi/(8kn)
\end{equation} 
for $\delta$
of (\ref{eqdlt}) (cf. part (i) of Lemma \ref{learc}),
where $\delta_-\approx \pi/(4kn)$ for large $n$.
Alternatively we can choose the centers $c_p$ at the midpoints of the arcs $\mathbb A_p$, 
and
 rotate both  arcs  $\mathbb A_p$  and
 centers $c_p$ for  $p=0,\dots,k-1$ by  
 a fixed angle on the  unit circle $\{z:~|z|=1\}$.
For  a proper choice of the angle,  part (i) of Lemma \ref{learc} 
ensures
that $\delta\ge 2\sin(\pi/(4kn))$ 
where $2\sin(\pi/(4kn) \approx \pi/(2kn)\ge 0.25\pi/n^2$ for large integers $n$.  
This would have decreased bound  (\ref{eqdelt}) by a factor of 2, but to
support the application of  Theorem \ref{thlrba}
throughout the merging process, we use about $2n$ centers $c_p$
overall.
This makes the same impact on the value 
$\delta_-$ as halving the length of the arcs $\mathbb A_p$
and brings us back to  bound (\ref{eqdelt}).
\end{remark}

Combine Theorem \ref{thss1} with bound (\ref{eqdelt})
and obtain the following result.

\begin{corollary}\label{colrbap}
At the $k$th stage of recursive merging, for $1<k<n$, every admissible block $\widehat N_{q}$
of the matrix 
$PC$ of Theorem \ref{thlrba}
can be $\epsilon$-approximated by a matrix  of rank $\rho$
provided that 
$\epsilon=\frac{4\theta^{\rho}}{(1-\theta)\delta\pi}$
for $\delta$ of Theorem \ref{thss1}.
For a constant $\theta$, $0<\theta<1$ 
this holds where $\rho=O(\log (\frac{1}{\delta\epsilon}))$
and consequently, by virtue of (\ref{eqdelt}), where
\begin{equation}\label{eqlrbap}
\rho=O(\log (n/\epsilon)).
\end{equation}
\end{corollary}


\begin{remark}\label{redual}
Equation  (\ref{eqctr}) implies that
Theorem \ref{thlrba} and consequently Corollary \ref{colrbap} 
and Theorem \ref{thcv} of the next section
can be immediately extended 
to the case where  $C=(\frac{1}{s_i-t_j})_{i,j=0}^{n-1}$, 
$s_i=e\omega^{i}$ for all $i$, and 
the choice of the knots  $t_j$ 
is unrestricted. One can apply the
 FMM techniques  of  \cite{GR87} and  \cite{CGR98}
toward relaxing our restrictions on the 
knot sets $\mathbb S$ or  $\mathbb T$ of a Cauchy matrix.
They would guarantee 
 numerical stability, 
unlike
our alternative techniques
in Section \ref{sgsk}.
\end{remark}

\begin{remark}\label{redlt}  The
lower bound $\delta_-\approx \pi/(2km)$ on $\delta$ of Remark \ref{redelt}
is overly pessimistic
for many dispositions of the knots $s_i$ on the complex plane.
For example, $\delta_-$ is a constant where the value $|s_i|$ is close to 1
 for no $i$,
whereas  $\delta_-\ge\pi/m$ where $s_i=\omega_{m}^i$, $i=0,\dots,m-1$.
Furthermore typically at most a small fraction 
 of all differences $c_{j,q}-s_i$ 
 has absolute values
close to the  bound $\delta_-$. 
For constant $\delta_-$  bound 
(\ref{eqlrbap})
on the $\epsilon$-rank 
decreases to
$\rho=O(\log (1/\epsilon))$, which is the case in
 \cite[Section 3]{DGR96}, where all knots $s_i$ and $t_j$ are real.
\end{remark}


\section{Fast approximate computations with  structured
 matrices and extensions}\label{svmv}


\subsection{Definitions and auxiliary results}\label{sdaux}


We need some additional definitions and basic results.
$\alpha (M)$ and $\beta(M)$ denote the
arithmetic cost of 
 computing 
the vectors
 $M{\bf u}$  
and  
$M^{-1}{\bf u}$, respectively,
maximized over all unit vectors ${\bf u}$
and minimized over all algorithms, and
we write $\alpha_{\epsilon}(M)=\min_{|E|\le\epsilon} \alpha (M+E)$ 
and $\beta_{\epsilon}(M)=\min_{|E|\le\epsilon} \beta (M+E)$
for a fixed small positive $\epsilon$.
The straightforward algorithm supports the following bound.

\begin{theorem}\label{thmbv}
$\alpha (M)\le 2(m+n)\rho-\rho-m$ for an $m\times n$ matrix $M$ given with its generating pair 
of a length $\rho$.
\end{theorem}


\begin{theorem}\label{thpert} (See \cite[Corollary 1.4.19]{S98} for $P= -M^{-1}E$.)
Suppose $M$ and $M+E$ are two nonsingular matrices of the same size
and $||M^{-1}E||=\theta<1$. Then
$||I-(M+E)^{-1}M||  \le \frac{\theta}{1-\theta}$ and
$||(M+E)^{-1}-M^{-1}||\le \frac{\theta}{1-\theta}||M^{-1}||$.
In particular $||(M+E)^{-1}-M^{-1}||\le 1.5\theta ||M^{-1}||$
if $\theta\le 1/3$.
\end{theorem}


\subsection{Fast approximate computations with
CV matrices}\label{scvc}


Unlike the case of HSS matrices Corollary \ref{colrbap}  bounds
numerical rank only for  off-diagonal blocks defined
column-wise, but not row-wise. We can still devise fast 
algorithms for such generalized HSS matrices 
approximating a CV matrix $C$
because these column-wise bounds
hold throughout
the process of recursive merging.  
Our bounds on the cost of approximate solution of 
linear systems of equations actually
require that the associated merging processes
of this and the previous sections
for an HSS approximation involve no singular 
or ill conditioned auxiliary matrices \cite{Pa}, and in particular
the input matrix should be nonsingular and well conditioned. In view of Remark \ref{renmr}
this is a serious restriction, which 
is extended to our algorithms for solving
structured linear systems
of equations in the next subsections.


\begin{theorem}\label{thcv} 
Assume an $n\times n$ CV matrix $C$, a positive $\epsilon<n$ ,
and $\rho$ of (\ref{eqlrbap}). Then  
 $\alpha_{\epsilon}(C)=O(n\rho\log (n)))$
and 
$\beta_{\epsilon}(C)=O(n\rho ^2\log (n))$, and so
$\alpha_{\epsilon}(C)=O(n\log^2(n)))$ 
and $\beta_{\epsilon}(C)=O(n\log^3 (n))$ where $\rho= O(\log (n))$, whereas
$\alpha_{\epsilon}(C)=O(n\log(1/\epsilon)\log (n)))$
and 
$\beta_{\epsilon}(C)=O(n\log^2(1/\epsilon)\log (n))$
where $\rho= O(\log (1/\epsilon))$. 
\end{theorem}
\begin{proof} 
Apply recursive merging  to the matrix $C$.
At its $j$th stage, $j=0,\dots,l-1$, for
$l\le \lceil\log_2n\rceil$,
compute
a permutation matrix $P^{(j)}$
and an $\epsilon$-approximation $C_{\epsilon}^{(j)}$ of the matrix $P^{(j)}C$ 
where all admissible blocks have ranks
at most $\rho$ 
for $\rho$ and $\epsilon$  invariant
 at all stages of the merging process and 
satisfying equation (\ref{eqlrbap}) (cf.  Corollary \ref{colrbap}). Clearly
$\alpha_{\epsilon}(C)\le 
\alpha(C_{\epsilon}^{(j)})$ for all $j$.
Let  $C'=C_{\epsilon}^{(l-1)}$ denote the matrix entering
the last merging stage,
 $C_{\epsilon}'=\begin{pmatrix}
\widehat \Sigma'_0  & \widehat N_1  \\
\widehat \Sigma''_0  &  \widehat \Sigma'_1  \\
\widehat N_0     &   \widehat \Sigma''_1
\end{pmatrix}$
where  
$\widehat \Sigma_p=\begin{pmatrix}
\widehat \Sigma'_p   \\
\widehat \Sigma''_p
\end{pmatrix}$  for $p=0,1$ denote
the two extended diagonal blocks and where $\rank(\widehat N_p)\le \rho$ 
for $p=0,1$.
It follows that $\alpha_{\epsilon}(C)\le 
\alpha(C'_{\epsilon})\le \sum_{p=0}^1(\alpha(\widehat \Sigma_p)
+\alpha(\widehat N_p))+n$.
Apply Theorem \ref{thmbv} and obtain that
$\sum_{p=0}^1
\alpha(\widehat N_p)\le 2n\rho$.
Recursively apply this argument to estimate 
$\alpha(\widehat \Sigma_p)$ for $p=0,1$ and obtain
$\alpha(C')\le n+4n\rho l + \sum_{j=0}^{l-1}\alpha(\widehat\Sigma_j)$
where $\widehat\Sigma_j$ denotes the matrix made up of the 
extended diagonal blocks at the $j$th merging, $j=0,\dots,l-1$ and
having at most $3nh$ entries for every $j$.
Choose $h=O(\rho)$  
and obtain that $\sum_{j=0}^{l-1}\alpha(\widehat\Sigma_j)=O(nl\rho)$
and consequently  $\alpha_{\epsilon}(C)\le 
\alpha(C')=O(n\rho \log (n))$. The claimed bound on $\beta_{\epsilon}(C)$
is supported by the algorithms of  \cite{CGS07} and \cite{XXG12}.
The algorithms have been proposed for HSS matrices, approximating
the special CV matrix $C_{1,\omega_{2n}}$, but close examination in \cite{Pa}
shows that they support the claimed cost bound for CV matrices.
\end{proof}



\subsection{Extension to
Van\-der\-monde matrices and their transposes}\label{svvtv}


\begin{theorem}\label{thcvv1}
Suppose  we are given a  vector
${\bf s}=(s_i)_{i=0}^{n-1}$ defining
an $n\times n$ Van\-der\-monde matrix
$V=V_{\bf s}$. 
Write $s_+=\max_{i=0}^{n-1}|s_i|$
 and 
$\bar\epsilon=(s_++1)\epsilon$ where
$\log(1/\epsilon)=O(\log n)$.
 Then 
$\alpha_{\bar\epsilon}(V)+\alpha_{\bar\epsilon}(V^T)=
O(n\rho\log (n))$
and
$\beta_{\bar\epsilon}(V)+\beta_{\bar\epsilon}(V^T)=O(n\rho^2\log (n))$
for  $\rho$ of (\ref{eqlrbap}).
\end{theorem}
\begin{proof}
Equations (\ref{eqvs})--(\ref{eqvs-t})
reduce the computations with the matrices $V_{\bf s}$
and  $V_{\bf s}^T$ to the same computations with a Cauchy matrix
$C_{{\bf s},f}$, which is a CV matrix for any $f$ such that $|f|=1$.
This enables us to extend Theorem \ref{thcv}
to the matrices $V_{\bf s}$
and  $V_{\bf s}^T$ except that we must adjust the approximation
bound $\epsilon$ of that theorem. Let us show that it  is sufficient
to change it to  $\bar\epsilon$. 
The matrices  $\diag(\omega^{-j})_{j=0}^{n-1}$,
$\diag(f^{-j})_{j=0}^{n-1}$, and
$\frac{1}{\sqrt n}\Omega=(\sqrt n\Omega^H)^{-1}$
and their inverses are unitary, and 
so multiplication by them makes 
no impact on the output error norms.
Multiplication by the matrix $\diag(s_i^n-f^n)_{i=0}^{n-1}$
can increase the value $\log_2(1/\epsilon)$ 
 by at most $\log_2(s_+^n+1)$,
whereas multiplication by its inverse 
can increase this value by at most $\log_2(\Delta)$ for
$\Delta=1/\max_{\{f:~|f|=1\}}\min_{i=0}^{n-1}|s_i^n-f^n|$.
We can ensure that $\Delta\le 2n$ 
by choosing a proper value $f$, and so
$\log_2(\Delta)\le 1+\log_2(n)$. Such an increase
 makes no impact on the asymptotic bounds of 
Theorem \ref{thcv}.
\end{proof}


\subsection{Extension to the classes of $\mathcal {CV}$,  $\mathcal {V}$ and $\mathcal {V^T}$}\label{sextcvvt}


Assume a Cauchy-like matrix $M$ 
of the class $\mathcal C$ 
represented with its displacement generator $(F,G)$
of a length $d$. 
Part (c) of
Theorem \ref{thdexpr} enables us to reduce  
 the approximation of $M$  to approximation of 
its basic  matrix $C$. We immediately obtain that

\begin{equation}\label{eqclk}
\alpha_{\epsilon'}(M)\le 2d+d\alpha_{\epsilon}(C)
~{\rm for}~ \epsilon'\le\sum_{j=1}^d |{\bf f}_j|~|{\bf g}_j|~\epsilon\le d~|F|~|G|~ \epsilon.
\end{equation}
Combine this estimate with Theorem \ref{thcv} provided that $M$ is a CV-like matrix
and obtain 

\begin{equation}\label{eqcvlk}
\alpha_{\epsilon'}(M)=O(dn\rho\log(n))
\end{equation}
for $\rho$,  $\epsilon'$ and $\epsilon'$ 
satisfying  (\ref{eqlrbap}) and (\ref{eqclk}).
To estimate $\beta_{\epsilon'}(M)$ note that
the ranks of the matrices
of local approximation
increase by at most a factor of $d$
in the transition to the matrix $M$
from its basic matrix $C$ in the expression of part (c) of
Theorem \ref{thdexpr}, whereas
the approximation norm bounds 
are defined according to the expressions for $\epsilon$ 
and $\epsilon'$ of (\ref{eqclk}). 
Theorem \ref{thpert} extends the latter error norm bounds 
to the case of the solution of linear systems with these matrices
as follows, 
\begin{equation}\label{eqepsl}
\epsilon''=\epsilon''(\epsilon',M)=O(n\epsilon'||M^{-1}||),
\end{equation}
and we 
obtain  the following expression,
\begin{equation}\label{eqcvlkb}
\beta_{\epsilon''}(M)=O(d^2n\rho^2\log(n))
\end{equation}
where
$\rho$,  $\epsilon'$, and $\epsilon''$ 
satisfy (\ref{eqlrbap}), (\ref{eqclk})  and  (\ref{eqepsl}).
Furthermore   
we 
 reduce the approximation of the matrices of the classes
$\mathcal V$ and $\mathcal {V^T}$ to the
approximation 
 of CV matrices
by applying at first parts $(v)$ and $(v^T)$ of
Theorem \ref{thdexpr} and then
 the algorithms supporting Theorem \ref{thcvv1}.


\subsection{Extension to the case of arbitrary knots} 
\label{sgsk}


Next we  
transform  matrix structure
 based on Theorem
 \ref{thdiag} to extend our 
approximation algorithms to computations
with Cauchy and Cauchy-like matrices 
of the class $\mathcal C$ with any set of knots, 
and we also estimate 
 the impact of the 
transformations on the approximation errors.

Suppose  that
 ${\bf s}$,
${\bf t}$
and  ${\bf u}$ 
denote three  vectors
of dimension $n$ and that
an $n\times n$ Cauchy-like matrix 
$M\in \mathcal C_{\bf s,t}$ is
given with a displacement generator $(F,G)$
of a length $d$.
Transform matrix
structures to reduce  
 the solution  
of a linear system $M{\bf x}={\bf u}$
to some
computations with CV-like matrices and 
multiplication  of the matrix $M$ by the 
vector ${\bf e}=(1,\dots,1)^T$.
Fix a scalar $e$, $|e|=1$, write
 $P=MC_{{\bf t},e}$ and ${\bf x}=C_{{\bf t},e}{\bf y}$, and note that
$P{\bf y}=${\bf u}, where
$P\in \mathcal C_{{\bf s},e}$ is a CV-like matrix
with a displacement generator $(F_P,G_P)$
of a length at most $d+1$,
$F_P=(F~|~M{\bf e})$
and $G_P=(C_{{\bf t},e}^TG~|~{\bf e})$.
By applying these techniques to the matrix $M^T\in \mathcal C_{\bf t,s}$
we can alternatively reduce the linear system 
$M{\bf x}={\bf u}$
to the computation of the products 
$M^T{\bf e}$ and to some
computations with CV-like matrices.
Likewise part (c) of Theorem \ref{thdexpr}
 reduces the
 approximation of  the vector
${\bf x}=M{\bf u}$ 
to the approximation of  the $d$ vectors
$C_{\bf s,t}{\bf v}_i$
for ${\bf v}_i=\diag({\bf g}_i)_{i=1}^d{\bf u}$,
${\bf g}_i=G{\bf e}_i$, and $i=1,\dots,d$,
and to $O(dn)$
additional flops,
provided that the matrix $M\in \mathcal C_{\bf s,t}$
is given with its displacement generator  $(F,G)$
of a length $d$.
We can compute a displacement generator
of a length at most 2 for the matrix 
$C'=C_{\bf s,t}C_{{\bf t},e}$,  of the 
class $\mathcal C_{{\bf s},e}$,
and then reduce the computation of 
the vector
${\bf x}$ to multiplication of the $\mathcal {CV}$
matrix $C'$
by the  
vector ${\bf z}$ satisfying the CV linear system
of equations $C_{{\bf t},e}{\bf z}={\bf u}$.
In all cases we reduce the original tasks to computations
with CV matrices
and  readily verify that
multiplication by the auxiliary CV matrix $C$
increases the  approximation 
error norm of the output
 by at most a factor of $||C||~||C^{-1}||$.
As we showed in Remark \ref{renmr} this upper bound is large
unless most of the knots of the CV matrix $C$ have absolute values near 1, 
 and so one may benefit from alternative
direct applications of the FMM to Cauchy matrices.


\subsection{Fast approximate computations with
polynomials and rational functions}\label{sprf}


Together with our equations (\ref{eqfhrf}),
(\ref{eqvs}), and
(\ref{eqvs-}), the following results link polynomial and rational interpolation and multipoint 
evaluation to each other, 
multiplication of Van\-der\-monde and Cauchy matrices by a vector, and
the solution of Van\-der\-monde and Cauchy linear systems of equations 
(cf. \cite[Chapter 3]{P01}).
By using this link we can extend our results
on Van\-der\-monde and Cauchy matrices to polynomial and rational
evaluation and interpolation, respectively.

\begin{theorem}\label{thevint}
(i) Let $p(x)=\sum_{i=0}^{n-1}p_ix^i$, ${\bf p}=(p_i)_{i=0}^{n-1}$,
 ${\bf s}=(s_i)_{i=0}^{n-1}$, and ${\bf v}=(v_i)_{i=0}^{n-1}$ .
Then the equations $p(s_{i})=v_i$ hold for $i=0,1,\dots,n-1$
if and only if $V_{\bf s}{\bf p}={\bf v}$.
(ii) For a rational function $v(x)=\sum_{j=0}^{n-1}\frac{u_j}{x-t_j}$
with $n$ distinct poles $t_0,\dots,t_{n-1}$ and for $n$
distinct scalars $s_0,\dots,s_{n-1}$, write 
${\bf s}=(s_i)_{i=0}^{n-1}$, ${\bf t}=(t_j)_{j=0}^{n-1}$,
${\bf u}=(u_j)_{j=0}^{n-1}$, ${\bf v}=(v_i)_{i=0}^{n-1}$.
Then the equations $v_i=v(s_i)$ hold  for $i=0,\dots,n-1$ if and only if
$C_{\bf s,t}{\bf u}={\bf v}$.
(iii) The equation $t(x)=\prod_{j=0}^{n-1}(x-t_j)=x^{n}+w(x)$, 
 for $n$ distinct knots $t_0,\dots,t_{n-1}$,
is equivalent
to the 
linear systems of $n$ equations
$w(t_j)=-t_j^n$ or 
 $n+1$  equations $t(0)=(-1)^n\prod_{j=0}^{n-1}t_j$, $t(t_j)=0$ for $j=0,\dots,n-1$
in both cases,
that is to polynomial interpolation for the vectors $(-t_j)_{j=0}^{n-1}$ and
$(t(0),0,\dots,0)^T$, respectively (cf. Example \ref{exln}).
\end{theorem}


\subsection{Functional transformations of matrix structures
and computations with generalized Cauchy matrices}\label{sftms}


Our algorithms of this section fall into the general framework 
of the FMM, and we can incorporate their
 modifications and extensions known in the FMM literature,
such as application of
Lagrange interpolation instead of Taylor expansion (cf.  \cite{DGR96}, \cite{B10})
 and the extension to generalized input classes.
Toward the latter extension,  consider Cauchy matrices $C_{\bf s,t}$ as 
a discrete
representation
of the function $\frac{1}{s-t}$, transform this function  
into various other functions of the variable $s-t$ 
such as  $a+\frac{b}{s-t-c}$, $\frac{1}{(s-t)^2}$, 
and $\ln (s-t)$, and arrive at various {\em generalized Cauchy matrices} such as
$\big (a+\frac{b}{s_i-t_j-c} \big )_{i,j=0}^{n-1}$, $\big (\frac{1}{(s_i-t_j)^2}\big )_{i,j=0}^{n-1}$, 
and $\big (\ln (s_j-t_j))_{i,j=0}^{n-1}$. 
We can readily extend to these matrices
 the FMM/HSS algorithms and complexity estimates
 (cf. \cite{GR87}, \cite{DGR96}).
In particular observe that $\frac{1}{(s-t)^2}=\frac{1}{(s-c)^2}\frac{1}{(1-q)^2}$, 
$\frac{az+b}{z-h}=a+\frac{b+ah}{z-h}$,
 and $\ln (s-t)=\ln (s-c)+\ln (1-q)$
 for $q=\frac{t-c}{s-c}$. Let us sketch an
 application to polynomial computations.


\begin{example}\label{exln}
Represent the polynomial $t(x)=\prod_{j=0}^{n-1}(x-t_j)$ of Section \ref{scvm} 
and part (c) of Theorem \ref{thevint} as 
$\exp(\sum_{j=0}^{n-1}\ln(x-t_j))$ and approximate its values 
$t(s_i)=\exp(\sum_{j=0}^{n-1}\ln(s_i-t_j))$
at
the $n$th roots of unity $s_i=\omega^i$ by using the FMM/HSS
techniques. Now apply IDFT to the computed approximations to the 
 values $v(\omega^i)=t(\omega^i)-1$ of the 
polynomial $v(x)=t(x)-x^n$ of a degree at most $n-1$ to 
 approximate the coefficients
of the polynomials $v(x)$ and consequently $t(x)$. 
\end{example}


\subsection{Extensions of Theorem \ref{thdgr} and
approximate solution of Toeplitz and Hankel linear systems of equations}\label{stls}


If all knots $s_i$ and $t_j$ of a Cauchy matrix $C_{\bf s,t}$ lie on the real line $\{z:~\Im(z)=0\}$,
then Theorem \ref{thdgr} 
 decreases the bounds of Theorem \ref{thcv} 
on $\alpha_{\epsilon}(C)$ and $\beta_{\epsilon}(C)$
by factors of $\log(\rho)$ and  $\log^2(\rho)$, respectively,
and decreases also the $\epsilon$-rank  $\rho$ 
of the off-diagonal blocks from order of $\log(n/\epsilon)$
to  $O(\log(1/\epsilon))$, bounding it just in terms of the
error tolerance $\epsilon$.
Next we apply functional transformations of Cauchy matrices to 
extend Theorem \ref{thdgr} to cover the cases where
 the $2n$ knots can lie on  
 any line or circle on the complex plane.  


\begin{theorem}\label{thanyl}
Theorem \ref{thdgr} holds 
for a Cauchy matrix $C_{\bf s,t}$ where all knots $s_i$ and $t_j$
lie on any
line on the complex plane.
\end{theorem}


\begin{proof}  
Begin with the following observations
where $a\neq 0$ and $c$ are two complex constants,

\begin{equation}\label{eqsbtr}
C_{\bf s,t}=C_{\bf s',t'}~
{\rm where}~s'_i=s_i-c,~t'_j=t_j-c~{\rm for~all}~i~{\rm and}~j, 
\end{equation}

\begin{equation}\label{eqscl}
C_{\bf s,t}=aC_{\bf s',t'}~
{\rm where}~a\neq 0,~s'_i=s_i/a,~t'_j=t_j/a~{\rm for~all}~i~{\rm and}~j, 
\end{equation}
Now suppose  all knots $s_i$ and $t_j$ lie on a
line obtained by rotating the real line by an angle $\phi$
followed by the shift by a complex $c$.
Define the new knots $s'_i=(s_i-c)/a$
and $t'_j=(t_j-c)/a$ for $a=\exp(\phi\sqrt{-1})$ and all $i$ and $j$.
They lie on the real line. Apply Theorem \ref{thdgr} 
to the matrix $C_{\bf s',t'}$, and apply equations (\ref{eqsbtr}) and (\ref{eqscl})
to extend the resulting approximations to the matrix $C_{\bf s,t}$.
\end{proof}

\begin{remark}\label{revndsc}
Equations (\ref{eqsbtr}) and (\ref{eqscl}) show low impact of 
shift and scaling of the knots of Cauchy matrices,  
in sharp contrast to the impact of shift and scaling 
of the knots of Van\-der\-monde matrices.
\end{remark}


\begin{theorem}\label{thanyc}
Assume a positive tolerance $\epsilon<1$ and an $n\times n$ 
 Cauchy matrix $C_{\bf s,t}$ with all $2n$ knots $s_i$ and $t_j$
lying on a circle
 on the complex plane. 
Then  (cf. Remark \ref{reftr})
(i) $\alpha_{\epsilon}(C)=O(n\log (1/\epsilon))$, whereas 
(ii) $\beta_{\epsilon}(C)=O(n\log(n/\epsilon))$.
\end{theorem}


\begin{proof}  
Combine
equations (\ref{eqsbtr}) and (\ref{eqscl}) 
to reduce the proof to the case where the $2n$ knots lie on the unit circle
$\{z:~|z|=1\}$. Then fix any complex $a$ such that
$|a|=1$ and recall that the function $\frac{z}{a}=1+\frac{2\sqrt {-1}}{z'-\sqrt {-1}}$
and its converse $z'=\frac{z+a}{z-a}\sqrt {-1}$ transform the real line into
this unit circle and vice versa.
Now  write  $s_i'=\frac{s_i+a}{s_i-a}\sqrt {-1}$
and $t_j'=\frac{t_j+a}{t_j-a}\sqrt {-1}$ for all $i$ and $j$
 and obtain that all knots $s'_i$ and $t'_j$ are real, whereas
 $s_i=a(1+\frac{2\sqrt {-1}}{s_i'-\sqrt {-1}})$,
$t_j=a(1+\frac{2\sqrt {-1}}{t_j'-\sqrt {-1}})$, 
$s_i-t_j=
2a\frac{s'_i-t'_j}{(s_i'-\sqrt {-1})(t_j'-\sqrt {-1})\sqrt {-1}}$, and
consequently 
$\frac{1}{s_i-t_j}=\frac{u_iv_j}{s'_i-t'_j}$ for $u_i=\frac{\sqrt {-1}}{2a}(s_i'-\sqrt {-1})$ and
$v_j=t_j'-\sqrt {-1}$. It follows that 
the Cauchy matrix $C=C_{\bf s,t}$ satisfies 
\begin{equation}\label{eqsclshr}
C=\diag(\widehat u_i)_{i=0}^{n-1}C_{\bf s',t'}\diag(v_j)_{j=0}^{n-1}~{\rm for}~
{\bf s'}=(s'_i)_{i=0}^{n-1}~{\rm and}~{\bf t'}=(t'_j)_{j=0}^{n-1}.
\end{equation} 
Apply Theorem \ref{thdgr} to the matrix $C_{\bf s',t'}$
and deduce that $\alpha_{\epsilon}(C)=O(n\log (1/\epsilon'))$ and
$\beta_{\epsilon}(C)=O(n\log(1/\epsilon'))$
 for $\epsilon'\le u~v~\epsilon$, $u=\max_{i=0}^{n-1}|u_i|$
and $v=\max_{j=0}^{n-1}|v_j|$.
Recall that $|a|=1$ and deduce that $|u_i|\le 0.5(|s'_i|+1)$
and $|v_j|\le (|t'_j|+1)$ for all $i$ and $j$. Recall that $|s_i|=|t_j|=1$
and so $|s'_i|\le 2/|s_i-a|$ and $|t'_j|\le 2/|t_j-a|$
for all $i$ and $j$. Choose a point $a$ on the unit circle
lying at the maximal distance $\delta\ge \delta_-=2\sin (0.25 \pi/n)$
from the set of the $2n$ knots
$\{s_0,\dots,s_{n-1},t_0,\dots,t_{n-1}\}$. It follows that $\delta_-\ge 2/n$ for $n>3$. 
Consequently 
$u\le 0.5(\frac{2}{\delta}+1)\le 0.5(n+1)$, $v\le \frac{2}{\delta}+1 \le n+1$,  and 
$\epsilon'\le 0.5(\frac{2}{\delta}+1)^2~\epsilon\le 0.5(n+1)^2~\epsilon$ for $n>3$, 
which implies the bounds
$\alpha_{\epsilon}(C)=O(n\log (n/\epsilon))$ and
$\beta_{\epsilon}(C)=O(n\log(n/\epsilon))$.

To decrease the bound on $\alpha_{\epsilon}(C)$,
partition the unit circle $\{z:~|z|=1\}$ into three semi-open arcs
$\mathcal A_h=\{\exp(\phi\sqrt{-1}):~2\pi h/3\le \phi<2\pi (h+1)/3\}$, 
each of length $2\pi/3$, for $h=0,1,2$,
and write 
$C_{h}=(c_{i,j}^{(h)})_{i,j=0}^{n-1}$, $c_{i,j}^{(h)}=0$ if $s_i,t_j\in \mathcal A_h$,
$c_{i,j}^{(h)}=\frac{1}{s_i-t_j}$ otherwise, for
$h=0,1,2$.   
We can estimate $\alpha_{\epsilon}(C_h)$ by applying our previous
argument, but now we choose $a=a(h)$ being the midpoint of the arc 
$\mathcal A_h$ and observe that in this case $\delta_-$ increases to 1,
which implies that $\epsilon'\le 0.5(\frac{2}{\delta}+1)^2~\epsilon\le 4.5~\epsilon$
and $\alpha_{\epsilon}(C_h)=O(n\log (1/\epsilon))$ for $h=0,1,2$. Finally observe that 
$C=C_0+C_1+C_2$  and obtain that  
$\alpha_{\epsilon}(C)\le \sum_{h=0}^2 \alpha_{\epsilon}(C_h)=
O(n\log (1/\epsilon))$.
\end{proof}


We immediately extend the bound on 
$\alpha_{\epsilon}(C)$
to the case of Cauchy-like matrices $M$ 
of part (c) of Theorem \ref{thdexpr} 
with the knots on a line or a circle. Namely we
combine equation (\ref{eqclk})
with part (i) of Theorem \ref{thanyc} 
and obtain that
\begin{equation}\label{eqcvlka}
\alpha_{\epsilon'}(M)=O(dn\log(1/\epsilon))
\end{equation}
where $\epsilon'$ and $\epsilon$ are linked by equation (\ref{eqclk}).
To estimate $\beta_{\epsilon}(M)$
for such matrices $M$ we first observe the bound of order
$O(\log(1/\epsilon))$ on the $\epsilon$-rank of the above matrices
$C_0$, $C_1$, $C_2$ and  $C_{\bf s,t}=C_0+C_1+C_2$.
It follows that $\beta_{\epsilon}(C_{\bf s,t})=O(n\log(n)\log^2(1/\epsilon))$
(cf. Theorem \ref{thcv}).
Apply the techniques of  Section \ref{sextcvvt}
to extend this estimate 
to  Cauchy-like matrices $M$ 
as follows.


\begin{corollary}\label{coclk}
Suppose $M$ is a Cauchy-like matrix of part (c) of Theorem \ref{thdexpr} 
having its knots on a line or a circle.
Then $\alpha_{\epsilon}(M)=O(nd\log(1/\epsilon'))$ 
and $\beta_{\epsilon''}(M)=O(nd^2\log(n)\log^2(1/\epsilon))$
for $\epsilon'$ and $\epsilon''$
defined in Section \ref{sextcvvt}.
\end{corollary}


Finally recall the canonical DFT-based transforms of the matrices of 
the classes $\mathcal T$ and  $\mathcal H$ into Cauchy-like matrices (see Theorem \ref{thtc})
and extend the corollary to obtain the following result.


\begin{corollary}\label{cotpl}
Suppose $M$ is a Toeplitz-like or Hankel-like matrix of parts (t)
or (h) of Theorem \ref{thdexpr}. 
Then  $\beta_{\epsilon''}(M)=O(nd^2\log(n)\log^2(1/\epsilon))$
for $\epsilon''$ and $\epsilon'$
defined in Section \ref{sextcvvt} and for $d\le 2$ in the case of Toeplitz
and Hankel matrices $M$.
\end{corollary}


\begin{remark}\label{reftr}
The transform of Cauchy matrices based on the function 
$\frac{az+b}{z-h}=a+\frac{b+ah}{z-h}$ is numerically unstable
for many values of the parameters $a$, $b$  and $h$, 
but  we 
avoid numerical problems by employing
the numerically stable algorithms of 
 \cite{CGS07},
\cite{X12},
 \cite{XXG12}, and
 \cite{XXCB} and  applying 
the functional transform just to bound the numerical ranks 
of the admissible blocks of Cauchy matrices involved into the computations. Then 
application of Theorem \ref{thcv} for $\rho=O(\log(1/\epsilon))$
still yields the cost bounds $\alpha_{\epsilon'}(M)=O(nd\log(1/\epsilon))$
and $\beta_{\epsilon''}(M)=O(nd^2\log(n)\log^2(1/\epsilon))$.
\end{remark}


\subsection{$\epsilon$-ranks of admissible blocks
and the impact on implementation}\label{ssmplimpl}


Our proof of Theorem \ref{thcv}
is  constructive, that is  
we  can readily compute the centers
$c_q$ and the admissible
blocks $\widehat N_q$ of bounded ranks
throughout the merging process,  
and then we can apply the algorithms 
of the previous section. 
In practice one should avoid a large part of these
computations, however,
by following the papers \cite{CGS07},
\cite{X12},
 \cite{XXG12}, and
 \cite{XXCB}.
They bypass the computation of the centers
$c_q$ and immediately
 compute the HSS generators for the
 admissible blocks $\widehat N_q$,
defined by HSS trees. The length of the generators can be 
chosen equal to the available  
upper bound $\rho$ on the numerical ranks of these blocks
or can be adapted empirically.
Theorem \ref{thcv}  implies that the computational cost bounds $\alpha_{\epsilon}(M)$ 
and $\beta_{\epsilon}(M)$
are proportional to $\rho$ and $\rho^2$, respectively,
and thus decrease as the numerical rank $\rho$ decreases.
If bound
(\ref{eqlrbap})
on the $\epsilon$-rank decreases
to
$\rho=O(\log (1/\epsilon))$
(cf. our Remark \ref{redlt}),
then the complexity bounds of Theorem \ref{thcv}
decrease to the level
$\alpha_{\epsilon}(C)=O(n\log(1/\epsilon) \log (n))$
and $\beta_{\epsilon}(C)=O(n\log^2(1/\epsilon) \log (n))$.
By virtue of Corollary \ref{cotpl}
this is the case for the
inputs from the classes $\mathcal T$ and  $\mathcal H$,
including the case of
Toeplitz and Hankel inputs, where such bounds have been empirically 
observed in \cite{XXG12}.


\section{Conclusions}\label{scnc}


The techniques of the transformation of matrix structures based on displacement 
representation  go back to \cite{P90}, with 
surprising algorithmic applications explored 
since 1995. At first we revisited these techniques
 covering them 
 comprehensively. We 
simplified their study
by employing Sylvester's (rather than Stein's) 
displacements and
the techniques for 
operating with them from
\cite{P00} and  \cite[Section 1.5]{P01}.
Then we covered some 
fast  numerically stable approximation algorithms based on combining
these transformations with another link among distinct 
classes of structured matrices, namely  among Cauchy and HSS matrices,
the latter matrices appeared inthe study of the FMM (that is the Fast Multipole Method).  
These  efficient algorithms
approximate the solution of a nonsingular Toeplitz or Toeplitz-like 
linear system of equations in nearly linear (versus classical cubic)
arithmetic time \cite{MRT05}, \cite{CGS07}, \cite{XXG12}, and \cite{XXCB}.
Our analysis of these algorithms 
revealed their additional power, and 
we extended them to 
support nearly linear arithmetic time bounds 
(versus known quadratic)
for the approximation of the matrix-by-vector products of Van\-der\-monde,
transposed Van\-der\-monde, and CV matrices, the latter ones being a subclass of Cauchy matrices   
and for approximate solution of nonsingular linear systems of equations with these 
matrices. We noted some potential numerical limitations for the application 
of our transformations to the latter task of solving linear systems of equations
and for our algorithmic transformations of Section \ref{sgsk}
from CV matrices to Cauchy matrices with any set of knots.
We observed no such limitations, however, in extension of our results to 
the matrices of the classes $\mathcal V$, $\mathcal V^T$, and  $\mathcal {CV}$, and
we further accelerated a little
the cited  numerical
approximation algorithms 
for  Toeplitz linear systems
 by combining the algorithms of 
\cite{DGR96} for polynomial evaluation with 
 functional transformations of matrix structures.

At this point natural research challenges include (i) the search for new
links and new transformations among various classes of structured matrices 
towards
significant algorithmic applications
(possibly by combining the displacement and functional 
transformations of matrix structures with the approximation 
techniques of the FMM)
 and (ii) the refinement 
 of the presented algorithms. Our arithmetic time bounds
for $(l,u)$-HSS computations exceed the bounds for similar
computations with banded matrices by logarithmic
factors, and one may try to close these gaps
by applying the advanced techniques of the FMM.
A more specific idea towards a specific goal
is the combination of equation (\ref{eqvtinv}) and Example \ref{exln}
in order to accelerate approximate solution of Van\-der\-monde
linear systems to the level achieved for multiplication
of a transposed Van\-der\-monde matrix by a vector.

Even the acceleration to the level of Theorem \ref{thdgr}, however,
would  support  substantially inferior Boolean complexity estimates
(with the excess by a factor of $\log(1/\epsilon)$)
compared to  
 the algorithms of \cite[Main Theorem]{BP87},
\cite[Theorem 3.9 and Section 5.3]{K98},
and \cite{PT13}
for high precision polynomial 
evaluation and interpolation
and consequently for the related Van\-der\-monde
and Cauchy
matrix computations. Equation (\ref{eqfhr})  
may help to extend the latter algorithms to high precision multiplication of a
Cauchy matrix by a vector and solving a Cauchy 
linear system of equations. Such progress may 
eventually become of
interest for numerical 
computations as well, in the case of 
 sufficient support from
the field of  Computer Arithmetic (cf. \cite{P91}).

\medskip

{\bf Acknowledgements:}
Our research has been supported by NSF Grant CCF--1116736 and
PSC CUNY Awards 64512--0042 and 65792--0043. 
We also wish to thank the reviewers for thoughtful valuable comments.

\end{document}